\documentclass[12pt]{article}
\usepackage[colorlinks]{hyperref}
\usepackage{amsmath,amsfonts,amsthm,amssymb,array,mathrsfs,latexsym,paralist, xypic, color,tabu,xypic}
\usepackage{amscd}
\usepackage{multirow}
\usepackage{yfonts,mathtools}
\usepackage{tikz, soul, qtree}
\usepackage{lipsum, stix}
\usepackage{tikz-cd}
\usepackage{dsfont}
\usepackage{graphicx,pdfpages}
\usepackage{lipsum, ulem}

\providecommand{\keywords}[1]{{\textit{Keywords and phrases:}} #1}
\providecommand{\classification}[1]{{\textit{2020 Mathematics Subject Classification:}} #1}

\newcommand{\Kantorovich}[1]{{\text{mk}_{#1}}}

\newcommand{\MongeKant}{{Mon\-ge-Kan\-to\-ro\-vich metric}}

\newcommand{\qcms}{quantum compact metric space}

\newcommand{\unit}{1}

\newcommand{\dom}[1]{{\operatorname*{dom}\left({#1}\right)}}

\newcommand{\worknote}[1]{}

\renewcommand{\geq}{\geqslant}
\renewcommand{\leq}{\leqslant}

\newtheorem{thm}{Theorem}[section]
\newtheorem{prop}[thm]{Proposition}
\newtheorem{cor}[thm]{Corollary}
\newtheorem{lemma}[thm]{Lemma}
\theoremstyle{remark}
\newtheorem{rmk}[thm]{Remark}
\newtheorem{example}[thm]{Example}
\theoremstyle{definition}
\newtheorem{defn}[thm]{Definition}

\newcommand{\bi}{\begin{itemize}}
	\newcommand{\ei}{\end{itemize}}
\newcommand{\be}{\begin{enumerate}}
	\newcommand{\ee}{\end{enumerate}}

\newcommand{\T}{\mathbb{T}}

\newcommand{\R}{\mathbb{R}}
\newcommand{\N}{\mathbb{N}}
\newcommand{\Z}{\mathbb{Z}}

\newcommand{\length}{\mathbb{L}_p^{\Sigma}}

\providecommand{\keywords}[1]{{\textit{Keywords and phrases:}} #1}
\providecommand{\classification}[1]{{\textit{2010 Mathematics Subject Classification:}} #1}

\def\IoIIdimdots(#1/#2/#3,#4){\node at (#1,#4) {$.$};\node at (#2,#4) {$.$};\node at (#3,#4) {$.$};}
\def\IIoIIdimdots(#1,#2/#3/#4){\node at (#1,#2) {$.$};\node at (#1,#3) {$.$};\node at (#1,#4) {$.$};}

\def\IoIIIdimdots(#1/#2/#3,#4,#5){\node at (#1,#4,#5) {$.$};\node at (#2,#4,#5) {$.$};\node at (#3,#4,#5) {$.$};}
\def\IIoIIIdimdots(#1,#2/#3/#4,#5){\node at (#1,#2,#5) {$.$};\node at (#1,#3,#5) {$.$};\node at (#1,#4,#5) {$.$};}
\def\IIIoIIIdimdots(#1,#2,#3/#4/#5){\node at (#1,#2,#3) {$.$};\node at (#1,#2,#4) {$.$};\node at (#1,#2,#5) {$.$};}

\usepackage[margin=1in]{geometry}
\AtEndDocument{\bigskip{\small%
		
		\addvspace{\medskipamount}
		\textsc{C. Farsi, J. Packer: Department of Mathematics, University of Colorado at Boulder, Boulder, CO 80309-0395, USA.}\par\textsc{T. Landry: The Fields Institute for Research in Mathematical Sciences, 222 College Street, Second Floor, Toronto, Ontario M5T 3J1, Canada.}\par \textsc{N.S. Larsen: Department of Mathematics, University of Oslo, P.O. Box 1053 Blindern, N-0316 Oslo, Norway.}\par
		\textit{C. Farsi e-mail address}: \texttt{carla.farsi@colorado.edu},  
		\textit{T. Landry e-mail address}: \texttt{tland004@ucr.edu},  
		\textit{N. Larsen e-mail address}: \texttt{nadiasl@math.uio.no}, 
		\textit{J. Packer e-mail address}: \texttt{packer@colorado.edu}.
}}

\begin{document}
	
\title{Spectral Triples for Noncommutative Solenoids and a Wiener's lemma}
\author {Carla Farsi, Therese Basa Landry, Nadia S. Larsen  and Judith A. Packer}
\date{\today}
\maketitle

\tableofcontents
	
	\begin{abstract} 
		In this paper we  construct odd finitely summable spectral triples based on length functions of bounded doubling on  noncommutative solenoids.  Our spectral triples induce a Leibniz Lip--norm on the state spaces of the noncommutative solenoids, giving them  the structure of Leibniz quantum compact metric spaces. By applying methods of R. Floricel and A. Ghorbanpour, we also  show that our odd spectral triples on noncommutative solenoids can be considered as direct limits of spectral triples on rotation algebras.   In the final section we prove a noncommutative Wiener's lemma and  show that our odd spectral triples can be defined to have an associated smooth dense subalgebra which is stable under the holomorphic functional calculus, thus answering a question of  B. Long and  W. Wu. The construction of the smooth subalgebra also extends to the case of nilpotent discrete groups.
		
			\end{abstract}
		
		\classification{46L87, 22D15,  58B34 (primary), and 47B07 (secondary)}.
		

\section{Introduction}

Our aim in this paper is to build and study noncommutative manifold structures (specifically spectral triples and their associated metric structures) on    noncommutative solenoid $C^*$-algebras, as well as to analyze some of their  associated smooth subalgebras. Our  methods also apply to commutative solenoids as well, even though we will not emphasize the commutative case here.

In the 1980's, A. Connes initiated a program to adapt classical tools from topology and Riemannian geometry to the operator algebraic setting.  For example, Connes showed in \cite{Con89, connes} that the geodesic distance $d_g$ on a closed spin Riemannian manifold $M$ can be recovered from the $C^*$-algebra $C(M)$, the Hilbert space $H$ of $L^2$-spinor fields, and the Dirac operator $D$, via, $\;\forall x,\;y\in M,$ 
\begin{equation}\label{def:MK-metric}
d_g (x, y) = \sup \left\{ |f(x) - f(y)| \, : \, f \in C(M),\,  || [D, f] ||_{\mathcal{B}(H)} \leq 1 \right\}.
\end{equation}
 These ideas were generalized to the context of noncommutative $C^*$-algebras, which    Connes formalized in the  definition of a \textit{spectral triple}. A spectral triple,  in its most basic form, is given by  a unital $C^*$-algebra $A$ represented on a Hilbert space $H,$ and a Dirac operator $D$ behaving well with respect to taking the commutators with a dense $\ast$-subalgebra of $A.$ Under appropriate circumstances, the  Monge-Kantorovich metric $mk_L$ on the state space ${\mathcal S}(A)$ of $A$ (which is defined generalizing Equation \eqref{def:MK-metric}) induces the weak$^*$-topology. In this paper, all spectral triples will be odd, that is, ungraded.

M. Rieffel considered, more generally and in absence of a spectral triple,  similar metric  structures on unital $C^*$-algebras (and in particular on  their state spaces) associated to seminorms; the resulting spaces  were named  {\it compact quantum  metric spaces}, see  \cite{Rieffel2}.    F. Latr\'emoli\`ere (see e.g. \cite{La}) considered a special case of Rieffel's spaces, and switched  terminology to {\it quantum compact metric spaces}, terminology which we will follow in this paper. 
Rieffel also became interested in the problem of constructing spectral triples on reduced group $C^*$-algebras and twisted reduced group $C^*$-algebras associated to  discrete groups  by using the left regular representation together with a Dirac operator constructed from a proper length function, following pioneering work by Connes \cite{Con89}. The class of groups where this approach could be used to derive not only spectral triples, but also examples of  compact quantum  metric spaces, was extended to include finitely generated nilpotent groups with a length function of {\it bounded doubling}, see  \cite{Rieffel4, Christ-Rieffel}.

In two recent papers \cite{Long-Wu1, Long-Wu2}, Long and Wu followed up on the work of Christ and Rieffel, by constructing odd spectral triples for  twisted group $C^*$-algebras $C^*(\Gamma,\sigma)$, in which $\Gamma$ is not necessarily finitely generated, but it is endowed with a  length function satisfying either the  {\it $\sigma$-twisted rapid decay} or the  {\it bounded $\theta$-dilation} properties. (In this paper we will call this latter condition {\it bounded ${\bf t}$-dilation}).  However, the question of whether or not one could construct length functions satisfying the desired conditions on discrete nilpotent groups that were not finitely generated remained an open one, even for the case of discrete abelian groups.

As far as this paper is concerned, the prime examples  of  twisted abelian group $C^*-$algebras in which the underlying abelian groups are  not finitely generated are the so  called {\it noncommutative solenoids}. 
Noncommutative solenoids were introduced and studied  by Latr\'emoli\`ere and  Packer in \cite{LaPa1}, \cite{LaPa} and  \cite{LaPa2}, and  can  be realized as the twisted group $C^*$-algebras $C^*(\mathbb Z\Big[\frac{1}{p}\Big]\times \mathbb Z\Big[\frac{1}{p}\Big], \sigma )$, for a certain multiplier $\sigma$
(a multiplier is a 2-cocycle with values in $\T$). Here, as throughout the rest of the paper, $p$ is a fixed prime number. 
 
Latr\'emoli\`ere and  Packer  showed that these $C^*$-algebras have the structure of Leibniz quantum compact metric spaces, and are the limit in the direct sense  and in the sense of the Gromov-Hausdorff propinquity of noncommutative rotation algebras. However the   question of constructing spectral triples on them was not considered in their papers. 

Since Floricel and  Ghorbanpour in \cite{FlGh}  discuss the construction of direct limits of odd spectral triples for unital $C^*$-algebras, it thus became a natural question to ask whether spectral triples as constructed  by Connes \cite{Con89} and Rieffel \cite{Rieffel4} (as well as by Long and Wu \cite{Long-Wu1, Long-Wu2}) could be combined with the methods of Floricel and Ghorbanpour \cite{FlGh} to endow noncommutative solenoids  with   spectral triple structures (of inductive limit type). This is one of the main motivating questions that brought us to write this paper.

For the sake of completeness, we also note that V. Aiello, D. Guido and T. Isola in  \cite{AGI-NC-Sol}  constructed semifinite spectral triples on periodic (but not general) rational noncommutative solenoids.  Moreover, via Gabor analysis methods, A. Austad and F. Luef  construct spectral triples on noncommutative solenoids in  \cite{Aust-Luef}.  However, neither of these approaches are entirely satisfactory for our purposes.  The work of \cite{AGI-NC-Sol} gives semifinite and not finite spectral triples, and does not apply when looking at aperiodic noncommutative solenoids, while the work of Austad and Luef does not provide specific examples of weights that satisfy their needed conditions to obtain any specific spectral triple on the noncommutative solenoid.

 Therefore, in this paper one of our main aims is to combine the methods of Christ and Rieffel \cite{Christ-Rieffel},	and then Long and Wu \cite{Long-Wu1, Long-Wu2} to construct a length function on $\mathbb Z\Big[\frac{1}{p}\Big]\times \mathbb Z\Big[\frac{1}{p}\Big]$ with the bounded doubling property or bounded ${\bf t}$-dilation property.  This allows us to construct the desired finitely summable spectral triples for  arbitrary noncommutative solenoids.  Our spectral triples define a quantum compact metric space structure on noncommutative solenoids as well.  The spectral triples we construct are also direct limits of spectral triples on rotation algebras in the sense of Floricel and Ghorbanpour \cite{FlGh}.

 In a more general setting, we also establish that for any countable discrete nilpotent (hence abelian, so that our results can be applied to $\mathbb Z\Big[\frac{1}{p}\Big]\times \mathbb Z\Big[\frac{1}{p}\Big]$) group $\Gamma,$ the $\ast$-subalgebra $H^{1,\infty}_{\mathbb L}(\Gamma,\sigma)$  is stable under the holomorphic functional calculus in $\ell^1(\Gamma, \sigma)$, and that this latter $\ast$-algebra is 
stable under the holomorphic functional calculus in $C^*(\Gamma,\sigma). $   We do this by adapting  work of P. Jolissaint in \cite{jol} together with  work of Austad \cite{Aust} involving the noncommutative Wiener Lemma of K. Gr\"ochenig and M. Leinert \cite{Groch-Lein}. In doing so, we  thus generalize some  results of I. Chatterji (see the appendix of  \cite{Mathai}), 
 which requires $\sigma$-twisted rapid decay.
 This will show that when $\Gamma$ is countable, discrete and nilpotent, $H^{1,\infty}_{\mathbb L}(\Gamma,\sigma)$ is stable under the holomorphic functional calculus in $C^*(\Gamma,\sigma).$   This also answers a question of Long and Wu from \cite{Long-Wu2} in the solenoid case.

 It remains to be seen whether the spectral triples presented here are limits of spectral triples on noncommutative rotation algebras  in the sense of Latr\'emoli\`ere (\cite{Lan-Lap-Lat, La22}).  
 
 The structure of our paper is as follows.  
 In Section 2, we review some preliminaries by giving the required definitions.  In Section 3, we construct the desired length function on $\Gamma=\mathbb Z[\frac{1}{p}]\times \mathbb Z[\frac{1}{p}]$ and show that the Dirac operator arising from this length function gives rise to a spectral triple on the noncommutative solenoid. We also show that we can use the structure of a  noncommutative solenoid as an inductive limit of noncommutative rotation algebras to view our spectral triple as one coming from an inductive system of spectral triples on noncommutative rotation algebras as defined by Floricel and Ghorbanpour \cite{FlGh}.  In Section 4, for any countable discrete  nilpotent group $\Gamma$ and multiplier $\sigma,$ we show that  the dense $\ast$-subalgebra $H^{1,\infty}_{\mathbb L}(\Gamma,\sigma)$ is stable  in $\ell^1(\Gamma,\sigma)$ under the holomorphic functional calculus, and that in turn $\ell^1(\Gamma,\sigma)$ is stable under the holomorphic functional calculus in $C^*(\Gamma,\sigma)$  (noncommutative Wiener's lemma). In the  case where the length function has bounded ${\bf t}$-dilation for some ${\bf t}>1,$ it is possible to use $H^{1,\infty}_{\mathbb L}(\Gamma,\sigma)$ as the dense smooth subalgebra associated to our spectral triples, thus answering the question of Long and Wu posed in \cite{Long-Wu2} in the solenoid case. 
 
At a late stage in the preparation of this manuscript, the authors became aware of a paper by P. Antonini, D. Guido, T. Isola, and A. Rubin \cite{AGIR}, in which they reconstructed the odd spectral triples of Long and Wu \cite{Long-Wu2} for twisted group $C^*$-algebras, as well as an even version of them. We  cite this paper for completeness.

 {\bf Acknowledgments:} This project began as a result of the \lq \lq Women in Operator Algebras II Workshop" held in BIRS, Banff, Canada, in December 2021.  
 
 This research was also  supported by the Simons Foundation (Simons Foundation Collaboration grants \#523991
 (C.F.) and \#316981 (J.P.)).  
 The fourth author would also like to thank the third author for her hospitality during her visits to the University of Oslo. 
 
 This material is based upon work supported by the National Science Foundation under Grant No. DMS-1928930 while the second author was in residence at the Mathematical Sciences Research Institute in Berkeley, California, during the Spring 2022 semester, as well as upon work supported by NSERC and the Government of Ontario while she was in residence at the Fields Institute for Research in Mathematical Sciences during the Summer and Fall of 2022.

\section{Preliminaries}

 We start by reviewing the definition of spectral triples.

\begin{defn}
[Connes \cite{Con89, connes}]
	\label{def:spectraltriple}
	A {\bf spectral triple} $(A, H, D)$ consists of a unital $C^*$-algebra $A$, a unital faithful representation $\pi$ of $A$ on a Hilbert space $H$, and a self-adjoint operator $D : \text{dom}(D) \subseteq H \rightarrow H$ such that
	\begin{itemize}
\item [(ST1)] the operator $D$ has compact resolvent $R_{\lambda}(D) =(D - \lambda \text{Id}_H)^{-1}$, $\lambda \in \mathbb{C} \setminus \sigma(D)$,
    
    \item [(ST2)] there exists a dense $\ast$-subalgebra ${\mathcal A}$ of $A$ such that for every $a\;\in\;{\mathcal A}$
    the commutator
    \[[D,\pi(a)]:= D\pi(a)-\pi(a)D\]
     is densely defined  and extends to a bounded operator on  $H.$  The subalgebra ${\mathcal A}$ is sometimes referred to as a {\bf smooth subalgebra} of $A$ with respect to $D.$
\end{itemize}
\noindent    
The operator $D$ is called a (generalized) \textbf{Dirac operator}. 
\end{defn}

Spectral triples generalize differential structure, as seen from the pioneering example of Connes of a  spectral triple associated to a spin manifold (in this case the spectral triple operator is the Dirac operator) \cite{Con89}.  The dense set in $A$ described in condition (ST2) is therefore analogous to the dense set of $C^\infty$ functions in the continuous functions $C(M)$ in the manifold case.  The compact resolvent condition ensures that the eigenvalues of $D$ exhibit properties that allow for the extraction of geometric information like measure and dimension from spectral data. 

The following is another example due to  Connes that is a special case of his result for more general group $C^*$-algebras. It was adapted to twisted group $C^*$-algebras for $\mathbb Z^d$ by Rieffel. 

\begin{example}(\cite{Con89},\cite{Rieffel4}))
\label{example_spectraltriplegroupalg}
To define a spectral triple for the group $C^*$-algebra $C^*(\mathbb Z^d),$ let $\|\cdot\|$ be any norm on $\mathbb R^d$ restricted to $\mathbb Z^d.$  Let $\pi$ be the regular representation of $C^*(\mathbb Z^d)$ on $\ell^2(\mathbb Z^d).$  Define the linear operator $D_{\|\cdot\|}$ on $\ell^2(\mathbb Z^d)$ by 
$$D_{\|\cdot\|}f(v)\;=\;\|v\|\cdot f(v),\;v\in\mathbb Z^d.$$
 The results of \cite{Rieffel4}, Theorem 0.1, show that $(C^*(\mathbb Z^d), \ell^2(\mathbb Z^d), D_{\|\cdot\|})$ is a spectral triple with smooth subalgebra $C_C(\mathbb Z^d):=\mathbb C\mathbb Z^d.$ This example extends to twisted group $C^*$-algebras $C^*(\mathbb Z^d,\sigma)$ as well.  
\end{example}
$\newline$

\subsection{Inductive Limit Spectral Triples}

E. Christensen and C. Ivan built spectral triples for $AF$-algebras \cite{ch-iv}, as did Connes in the particular case of the commutative $C^*$-algebra of all continuous functions on a Cantor set \cite{connes}.  Though slightly different, both  constructions of spectral triples have $C^*$-algebras $A$, Hilbert spaces $H$, and Dirac operators $D$, which for some directed set $(J, \leq)$, can each, respectively, be written as inductive limits of $C^*$-algebras $ \displaystyle{\lim_{\substack{\longrightarrow \\ j \in J}}} \, A_j$, Hilbert spaces $\displaystyle{\lim_{\substack{\longrightarrow \\ j \in J}}} \, H_j$, and self-adjoint operators  $\displaystyle{ \lim_{\substack{\longrightarrow \\ j \in J}}} \, D_j$.  Furthermore, $\{ (A_j, H_j, D_j) \}_{j \in J}$ is a family of spectral triples for which maps can be defined making this collection into an inductive system with inductive limit $(A, H, D)$.  Floricel and  Ghorbanpour formalized a framework which includes these inductive systems of spectral triples, as we will review below.
 
 The spectral triples for solenoids that we will construct in this paper will also be of this inductive limit kind.  In the Definition below, we slightly modify the  definition of Floricel and Ghorbanpour to allow choices of different smooth subalgebras.

\begin{defn}
[Floricel-Ghorbanpour \cite{FlGh}]
	\label{def:FGmorphism}
	A {\bf morphism between two spectral triples} $(A_1, H_1, D_1)$ and $(A_2, H_2, D_2)$ with respective smooth subalgebras ${\mathcal A}_1$ and ${\mathcal A}_2$ is a pair $(\phi, I)$ consisting of a unital $^*$-homomorphism $\phi: A_1 \rightarrow A_2$ and a bounded linear operator $I: H_1 \rightarrow H_2$ satisfying the following conditions:
	
	(1) $\phi( {\mathcal A}_1)\subseteq {\mathcal A}_2$, where ${\mathcal A}_1$ and ${\mathcal A}_2$ are as in $(ST2)$,
	
	(2) $I\pi_1(a) = \pi_2(\phi (a) ) I$, for every $a \in A_1$,
	
	(3) $I(\text{dom}(D_1)) \subseteq \text{dom}(D_2)$ and $I D_1 = D_2 I$.
	$\newline$
A morphism $(\phi, I)$ is said to be {\bf isometric} if $\phi$ is injective and $I$ is an isometry.
\end{defn}

\begin{defn}
[Floricel-Ghorbanpour \cite{FlGh}]
	\label{def:FGinductivelim}
	Let $(J, \leq)$ be a directed index set and let $\{ (A_j, H_j, D_j) \}_{j \in J}$ be a family of spectral triples, with respective smooth subalgebras $\{{\mathcal A_j}\}_{j \in J}.$  Suppose that for every $j, k \in J$, $j \leq k$, an isometric morphism $(\phi_{j,k}, I_{j,k})$ from $(A_j, H_j, D_j)$ to $(A_k, H_k, D_k)$ is given such that $\phi_{k,l} \phi_{j,k} = \phi_{j,l}$ and $I_{k,l} I_{j,k} = I_{j,l}$ for all $j,k,l \in J$, $j \leq k \leq l$.  The resulting system $\{ (A_j, H_j, D_j), (\phi_{j,k}, I_{j,k}) \}_J,$ with respective smooth subalgebras $\{{\mathcal A_j}\}_{j \in J},$ is called an {\bf inductive system of spectral triples}.
\end{defn}

\begin{defn}
[Floricel-Ghorbanpour \cite{FlGh}]
	\label{def:FGspectraltriple}
     Let $\{ (A_j, H_j, D_j), (\phi_{j,k}, I_{j,k}) \}_J$ be an inductive system of spectral triples, with respective smooth subalgebras $\{{\mathcal A_j}\}_{j \in J},$ with 
     $$A \;=\; \lim_{\substack{\longrightarrow \\ j \in J}} A_j, \quad H \;=\; \lim_{\substack{\longrightarrow \\ j \in J}} H_j, \quad $$
    $\newline$
    and inductive limit isometric operators $I_j: H_j \rightarrow H$, and the $^*$-monomorphisms $\phi_j: A_j \rightarrow A$. Let 
    $$\pi \;=\; \lim_{\substack{\longrightarrow \\ j \in J}} \pi_j $$
    $\newline$
    be the inductive limit of the family of representations $\{ \pi_j \}_{j \in J}$ associated with the family of spectral triples $\{ (A_j, H_j, D_j) \}_{j \in J}$. Let 
    $${\mathcal A}\;=\;\lim_{\substack{\longrightarrow \\ j \in J}} {\mathcal A}_j,$$
    where the second direct limit is taken in the sense of algebra, without completion.
      For every $\xi \in \bigcup_{j \in J} I _j(\text{dom}(D_j))$ of the form $\xi = I_j \xi_j$, $\xi_j \in \text{dom}(D_j)$, define
    $$ D \xi = I_j D_j \xi_j .$$
    $\newline$
    Then $(A, H, D)$ is called the {\bf inductive realization} of the inductive system $\{ (A_j, H_j, D_j), (\phi_{j,k}, I_{j,k}) \}_J$.  It has corresponding subalgebra ${\mathcal A},$ which is norm dense in $A.$
\end{defn}

Not every inductive realization of spectral triples gives rise to a spectral triple, because it is necessary that the candidate for the Dirac operator $D$ satisfies Condition (ST1) for spectral triples, and that the candidate ${\mathcal A}$ for the smooth subalgebra of inductive limit algebra satisfy Condition (ST2) for spectral triples.  However, in two theorems, Floricel and Ghorbanpour give necessary and sufficient conditions for an inductive realization of spectral triples to give rise to a spectral triple, as follows:

\begin{thm} [\cite{FlGh} ,Theorem 3.1 ]
\label{mainthm-FG}	
Let $\{ (A_j, H_j, D_j), (\phi_{j,k}, I_{j,k}) \}_J$ be an inductive system of spectral triples, with respective smooth subalgebras $\{{\mathcal A_j}\}_{j \in J},$ and 
with inductive realization $(A, H, D)$ and corresponding norm-dense subalgebra ${\mathcal A}$ as in Definition \ref{def:FGspectraltriple}.  The following conditions are equivalent.
\begin{itemize}
\item [(i)] $D$ has compact resolvent;
\item [(ii)] the sequence $\{I_j R_{\lambda}(D_j)I_j^*\}_{j\in\mathbb N}$ converges uniformly to $R_{\lambda}(D)$ for every $\lambda\in\mathbb C\backslash \mathbb R;$
\item [(iii)] the sequence $\{I_j R_{\lambda}(D_j)I_j^*\}_{j\in\mathbb N}$
converges uniformly to $R_{\lambda}(D)$ for some $\lambda\in\mathbb C\backslash \mathbb R;$
\item [(iv)] the sequence $\{I_j f(D_j)I_j^*\}_{j\in\mathbb N}$
converges uniformly to $f(D)$ for every continuous function $f$ on $\mathbb R$ vanishing at infinity.
\end{itemize}
\end{thm}

 Notice in Theorem \ref{mainthm-FG} that if the Dirac operators $\{D_j\}_{j\in\mathbb N}$ are diagonalizable,  a necessary condition for any one of the equivalent conditions of this Theorem to be satisfied is that for any $r>0,$ the closed disk around $0$ of radius $r$ in the complex plane should only contain finitely many eigenvalues of the $\{D_j\},$ where the eigenvalues are counted with multiplicity, or else it will be impossible for $D$ to have compact resolvent.   
Here is the second main Theorem and its fundamental corollary of Floricel and Ghorbanpour, which guarantees that Condition (ST2) is satisfied:

\begin{thm}
\label{mainthm2-FG}	
Let $\{ (A_j, H_j, D_j), (\phi_{j,k}, I_{j,k}) \}_J$ be an inductive system of spectral triples, with respective smooth subalgebras $\{{\mathcal A_j}\}_{j \in J},$ and let ${\mathcal A}\;=\;\lim_{j\to \infty} {\mathcal A_j}$ be the algebraic direct limit, and $D$ the corresponding candidate for Dirac operator given in Definition \ref{def:FGspectraltriple}.  Then the operator $[D, \pi(\phi_j(a))]$ is bounded if and
only if the family of operators $\{[D_k, \pi_k(\phi_{j,k}(a))]\}_{k\geq j}$ is uniformly bounded.
\end{thm}

\subsection{Noncommutative Solenoids as Twisted Group $C^*$-Algebras and Inductive Limits }
\label{sec:twistgpalgs}

Throughout this paper, all discrete group we consider will be countable.
We have mentioned earlier that group $C^*$-algebras of countable discrete groups equipped with length functions have been studied as noncommutative spaces by Connes, Rieffel, and others \cite{Con89, Rieffel3, Christ-Rieffel}.   Connes used certain  length functions on finitely generated groups to define metrics on the state space of the group $C^*$-algebras.  Building on the work of Connes, Rieffel investigated when this same construction yields a metric that agrees with the weak$^*$-topology of the state space. Rieffel also investigated this construction for twisted group $C^*$-algebras.  In order to continue these studies, we first review the construction of twisted group $C^*$-algebras, and then concentrate on the specific examples given by  noncommutative solenoids.

Let $\sigma$ be any multiplier of the countable discrete  group $\Gamma$,   i.e., a $2$-cocycle on $\Gamma$  taking values in $\mathbb T$.  For any $f_1, f_2 \in \ell^1(\Gamma)$, the twisted convolution $*_{\sigma}$ is given by

$$ f_1 *_{\sigma} f_2 : \gamma \in \Gamma \mapsto \sum_{\gamma_1 \in \Gamma} \, f_1(\gamma_1)f_2((\gamma_1)^{-1}\gamma )\, \sigma(\gamma_1, (\gamma_1)^{-1}\gamma  ),$$

$\newline$
and the adjoint operation by

$$ f_1^* : \gamma \in \Gamma \mapsto \overline{f_1( \gamma^{-1})\, \sigma(\gamma,  \gamma^{-1})}.$$  

\begin{defn}
\label{def:twistedgroupC*alg}
Given a discrete group $\Gamma$ and a multiplier $\sigma$ on $\Gamma,$ we define the {\bf left-$\sigma$ regular representation} $\lambda_{\sigma}$ of the group algebra 
 $\ell^1(\Gamma,\sigma)$ on $\ell^2(\Gamma)$ by, for all $f\;\in\;\ell^1(\Gamma,\sigma)$, $g\;\in\;\ell^2(\Gamma),$ $\gamma_1,\;\gamma\;\in\;\Gamma$:
$$\Big(\lambda_{\sigma}(f)(g)\Big)(\gamma)\;=\;\sum_{\gamma_1\in\Gamma}\sigma(\gamma_1,\gamma_1^{-1}\gamma)f(\gamma_1)g(\gamma_1^{-1}\gamma).$$
By restricting $\lambda_{\sigma}$ to the delta functions $\{\delta_{\gamma}:\gamma\in \Gamma\},$ we obtain the left-$\sigma$ regular representation for the group $\Gamma.$
The $C^*$-completion of $\ell^1(\Gamma,\sigma)$ in the norm given by the representation $\lambda_{\sigma}$ is called  {\bf the reduced twisted group $C^*$-algebra associated to $\sigma$} and denoted by 
$C^*_r\big(\Gamma, \sigma \big).$  The {\bf full twisted group algebra associated to $\sigma$}  is defined to be the $C^*$-enveloping algebra of the Banach algebra $\ell^1(\Gamma,\sigma)$ given the standard $\ell^1$--norm, and is denoted by $C^*\big(\Gamma, \sigma \big).$ It follows from the definition that $C^*\big(\Gamma, \sigma \big)$ is the universal $C^*$-algebra generated by all projective representations of $\Gamma$ with multiplier $\sigma:$ that is, by all families of unitaries $(W_{\gamma})_{\gamma \in \Gamma}$ such that for any $\gamma_1, \gamma_2 \in \Gamma$, $W_{\gamma_1}W_{\gamma_2} = \sigma(\gamma_1, \gamma_2)W_{\gamma_1 \gamma_2}.$   In the general case, $C^*_r\big(\Gamma, \sigma \big)$ is a quotient of $C^*\big(\Gamma, \sigma \big),$ but if $\Gamma$ is amenable, we have 
$C^*\big(\Gamma, \sigma \big)\;\cong\;C^*_r\big(\Gamma, \sigma \big)$ \cite{ZellerMeier}. When $\sigma=1$ (untwisted case), $C^*\big(\Gamma \big):= \;C^*\big(\Gamma, 1 \big)$ and $C^*_r\big(\Gamma \big):= \;C^*_r\big(\Gamma, 1 \big)$. We will also denote by 
\[
\label{eq:def-of-CC}C_C(\Gamma, \sigma)\]
the (dense) $\ast$-subalgebra of  $\ell^1(\Gamma, \sigma)$ of complex-valued functions of finite support on $\Gamma$.
\end{defn}

We will now describe some examples of multipliers we will use throughout the paper. First,  fix an arbitrary prime $p,$ let $\mathbb{Z}\Big[\frac{1}{p}\Big]$ denote the subgroup of rational numbers of the form- $ \mathbb{Z}\Big[\frac{1}{p} \Big] = \Big\{\frac{a}{p^k} \in \mathbb{Q} : a \in \mathbb{Z}, k \in \mathbb{N} \Big\}$ (endowed with the discrete topology).

\begin{example}
\begin{enumerate}
\item Fix  $\theta \in [0,1],$ define $\sigma_\theta$ on $\Z^2$ by:
$$ \sigma_{\theta} : \Z^2 \times \Z^2 \ni \Big( \begin{pmatrix} z_1 \\ z_2\end{pmatrix}, \begin{pmatrix} y_1 \\ y_2\end{pmatrix} \Big) \, \mapsto \, \text{exp}(\pi i \theta (z_1y_2-z_2 y_1)),$$
\item Let, for $p$ a fixed prime: 
\begin{equation}\label{eq:def-Omega}\Omega_p\;:=\;\{\theta=(\theta_n)\;\in\; \Pi_{n=0}^{\infty}[0,1)_n:\;\forall n\in \mathbb N  ,\;p\theta_{n}=\theta_{n-1}\;\text{mod}\;\mathbb Z\}.
\end{equation} 
Fix $\theta\in \Omega_p $,  and 
define a multiplier $ \sigma_{\theta}$ on $\Gamma=\mathbb{Z}\Big[\frac{1}{p}\Big]\times\mathbb{Z}\Big[\frac{1}{p}\Big]$ by:
\begin{equation}
\label{eq:bicharacter}
\sigma_{\theta} :  \Big( \mathbb{Z}\Big[\frac{1}{p}\Big]\times\mathbb{Z}\Big[\frac{1}{p}\Big]\Big)^2 \ni  \Big( \Big( \frac{q_1}{p^{k_1}}, \frac{q_2}{p^{k_2}}  \Big),  \Big( \frac{q_3}{p^{k_3}}, \frac{q_4}{p^{k_4}}  \Big) \Big) \mapsto \text{exp}(2 \pi i \theta_{k_1 + k_4}q_1q_4).
\end{equation} 
\end{enumerate}
\end{example}

We recall the following; the second definition is adapted from Definition 3.1 of \cite{LaPa}.

\begin{defn}
\label{def:ncsolenoid}
For $\theta \in [0,1],$ define $C^*(\Z^2, \sigma_{\theta})$ to be the {\bf  noncommutative two-torus of angle $\theta$;} this $C^*$-algebra is also called the rotation algebra $A_\theta$.   
That is, $C^*(\Z^2, \sigma_{\theta})\cong A_\theta$  is
the universal $C^*$-algebra generated by unitaries $U_{\theta}$ and $V_{\theta}$ such that
\begin{equation}
\label{eq:irratrotgens}
U_{\theta}V_{\theta} = e^{2 \pi i \theta} V_{\theta}U_{\theta}.
\end{equation} 
For $\theta \in \Omega_p$, the twisted group $C^*$-algebra $C^\ast\Big(\Gamma,\sigma_\theta\Big),$ is called the {\bf  noncommutative $\theta$-solenoid} and is denoted by $\mathcal{A}_\theta^{\mathcal S}$.
\end{defn}

We now recall from Theorem 3.6 of  \cite{LaPa} how noncommutative solenoids can be viewed as direct limits of noncommutative rotation algebras.
\begin{thm} \cite{LaPa}
	\label{thm:DLTorus}
	Let $p\in\mathbb N$ be prime, and fix $\theta\in\Omega_p$. For all $n\in\N$, let $\varphi_n$ be the unique *-morphism from $A_{\theta_{2n}}$ into $A_{\theta_{2n+2}}$ extending:
	\begin{equation*}
	 \varphi_n  \left\{
	\begin{array}{lcr}
	U_{\theta_{2n}} &\longmapsto& U_{\theta_{2n+2}}^p\\
	V_{\theta_{2n}} &\longmapsto& V_{\theta_{2n+2}}^p\\
	\end{array}.
	\right.
	\end{equation*}
	Then the direct limit $C^*$-algebra obtained from the above embeddings:
\[\lim_{\substack{\longrightarrow \\ n \in \N}}A_{\theta_{2n}}:\;\;
 A_{\theta_0}\; \stackrel{\varphi_0}{\longrightarrow}\; A_{\theta_2}\;\stackrel{\varphi_1}{\longrightarrow}\;A_{\theta_4} \;\stackrel{\varphi_2}{\longrightarrow}\; \cdots\; A_{\theta_{2n}}\;\stackrel{\varphi_n}{\longrightarrow}\;A_{\theta_{2n+2}}\cdots\;
\]
is $\ast$-isomorphic to  $\mathcal{A}_\theta^{\mathcal S}$  (of Definition \ref{def:ncsolenoid}). 
\end{thm}

We briefly give the outline of the proof for Theorem \ref{thm:DLTorus}.  By the natural embeddings of subgroups $\frac{1}{p^n} \mathbb{Z} \times \frac{1}{p^n} \mathbb{Z}$ into $\Gamma=\mathbb{Z}\Big[\frac{1}{p}\Big]\times\mathbb{Z}\Big[\frac{1}{p}\Big],$ the noncommutative solenoid $\mathcal{A}_\theta^{\mathcal S}=C^*(\Gamma, \sigma_{\theta})$ can be viewed as a direct limit of the $C^*$-algebras $\{ C^*(\frac{1}{p^n} \mathbb{Z} \times \frac{1}{p^n} \mathbb{Z}, (\sigma_{\theta})_n) \}_{n \in \mathbb{N}},$ where here by $ (\sigma_{\theta})_n,$ we mean the restriction of the multiplier $\sigma_\theta$ from $\Gamma$ to its subgroup $\frac{1}{p^n} \mathbb{Z} \times \frac{1}{p^n} \mathbb{Z}$ for a fixed $n\in\N.$ Then it is not hard to see that for every $n\in \mathbb N,$ we have $C^*(\frac{1}{p^n} \mathbb{Z} \times \frac{1}{p^n} \mathbb{Z},  (\sigma_{\theta})_n)\;\cong\;A_{\theta_{2n}}\;\cong\;C^*(\mathbb Z^2,\;\sigma_{\theta_{2n}}),$ because both of these $C^*$-algebras have two generating unitaries satisfying the fundamental commutation relation given in Equation \eqref{eq:irratrotgens}.  With these identifications, it easily follows that $\mathcal{A}_\theta^{\mathcal S}$ satisfies the universal properties that make it $\ast$-isomorphic to $\displaystyle{\lim_{\substack{\longrightarrow \\ n \in \mathbb N}}} A_{\theta_{2n}}.$ We will make these identifications frequently in upcoming sections.

In \cite{AGI-NC-Sol}, noncommutative solenoids of periodic type, such as the following example, were given the structure of noncommutative manifolds by Aiello, Guido, and Isola.  This was done by constructing a semifinite spectral triple on them. 

\begin{example} [Aiello-Guido-Isola: NC solenoid, periodic rational case]
	\label{example_rational}  
	Let $p=2.$   Consider $\theta= (\theta_n)\in \Omega_2$ given by:
	$$\theta=\;\big(\frac{2}{3},\frac{1}{3},\frac{2}{3}, \cdot \cdot \cdot \cdot \cdot \cdot \cdot\big).
	$$
	   By Theorem \ref{thm:DLTorus} $\mathcal{A}_{\theta}^{\mathcal S}$  is the direct limit of the single rotation algebra $A_{\frac{2}{3}}.$
\end{example}

 In Theorem 2.9 of \cite{LaPa} it is shown that for $\theta\in\Omega_p,\;\mathcal{A}_{\theta}^{\mathcal S}$ is simple if and only if given any distinct $j, k \in \mathbb{N}$, we have $\theta_j \neq \theta_k$.  In particular, the construction given by Aiello, Guido, and Isola in \cite{AGI-NC-Sol} cannot be applied to simple noncommutative solenoids like the one below.

\begin{example}  \label{example_rational non stable}  
	Let $a$ and $q$ be relatively prime integers and $p$ a prime. Consider
	\[\theta_0= a / q,\ \theta_1= a /(p  q),\ \theta_2= a /(p^2  q),\ldots, \theta_{2n}= a /(p^{2n}  q),\ldots. \]
	The resulting noncommutative solenoid $\mathcal{A}_\theta^{\mathcal S}$ is $\ast$-isomorphic to the direct limit of non-isomorphic rotation algebras corresponding to rational rotations.  
\end{example}
$\newline$
In Section 3 we will  define spectral triples on all noncommutative solenoids.

$\newline$

\section{Spectral Triples for Noncommutative Solenoids, and Their Quantum Compact  Metric Space Structure}
\label{sec:sptrNCS}

We now come to one of the  main results of the paper,  the construction of  spectral triples on noncommutative solenoids. (As a particular case, these results also hold for commutative solenoids.) Their associated metric structures in turn will give a new way of viewing noncommutative solenoids as  quantum compact metric spaces.

\subsection{Length Functions on $\Z[\frac1p] \times \Z[\frac1p]$ with Bounded Doubling}
	\label{subsec:lengthfunctionsonZp}
In this section we  define a length function on $ \Z[\frac1p]$ and $\Z[\frac1p] \times \Z[\frac1p]$. This will be used  later in our construction of spectral triples on noncommutative solenoids.

We recall for the reader's convenience the definition of {\it length function} on a general discrete  group.
\begin{defn}
	[Connes \cite{Con89}, Christ-Rieffel \cite{Christ-Rieffel}, Long-Wu \cite{Long-Wu2}]
	\label{def:lengthfcn}
	A \textbf{length function} on a discrete  group $\Gamma$ is a function $\mathbb{L}: \Gamma \to [0, \infty)$ such that
	
	(1) $\mathbb{L}(\gamma) = 0$ if and only if $\gamma = e$, where $e$ is the identity of $\Gamma$,
	
	(2) $\mathbb{L}(\gamma) = \mathbb{L}(\gamma^{-1})$ for all $\gamma \in \Gamma$,
	
	(3) $\mathbb{L}(\gamma_1 \gamma_2) \leq \mathbb{L}(\gamma_1)  +\mathbb{L}(\gamma_2)$ for all $\gamma_1, \gamma_2 \in \Gamma$.
\end{defn}

Given a pair consisting of a discrete group $\Gamma$ and a length function $\mathbb{L}$ on it, 
Connes and Rieffel associate  to $(\Gamma, \mathbb{L})$ a \lq \lq Dirac operator" by defining $D_{\mathbb{L}} : C_C(\Gamma) \rightarrow C_C(\Gamma)$ given by, for all $f \in C_C(\Gamma), \gamma \in \Gamma$:

$$D_{\mathbb L}(f)(\gamma)\;=\;\mathbb L(\gamma)f(\gamma).$$  As mentioned in the introduction, for many examples of  $(\Gamma,\mathbb L)$ (mostly with finitely generated $\Gamma$), given a multiplier $\sigma$ on $\Gamma,$ and corresponding  left-$\sigma$ regular representation $\lambda_{\sigma}$ of $C^*(\Gamma,\sigma)$ on $\ell^2(\Gamma),$  $(C^*(\Gamma,\sigma), \ell^2(\Gamma),D_{\mathbb L} )$ has been shown to be a spectral triple with smooth subalgebra  $C_C(\Gamma,\sigma).$

\begin{defn}
	\label{def:lengthfcnbddoubl}
A length function $\mathbb{L}$ on a countable discrete group $\Gamma$ 
(as in Definition \ref{def:lengthfcn}) is in addition: 
\begin{enumerate}
    \item [(1)] \textbf{proper} if $B_{\mathbb{L}}(R) := \{\gamma \in \Gamma : \mathbb{L}(\gamma) \leq R \}$ is a finite subset of $\Gamma$ for each $0\leq R < \infty;$
    
   \item [(2)] has (or is of) \textbf{bounded doubling} if $\mathbb{L}$ is proper and there exists a constant $C_{\mathbb{L}} < \infty$ such that 
    
    $| B_{\mathbb{L}}(2R) | \leq C_{\mathbb{L}} | B(R)|$ for each $R,\;1\;\leq\;R\;<\infty.$
\end{enumerate}
\end{defn}

We introduce for future reference the concept of  bounded ${\bf t}$-dilation for length functions on discrete groups, first introduced by B. Long and W. Wu \cite{Long-Wu2}.

\begin{defn} (c.f. \cite{Long-Wu2}, Definition 5.1)
	\label{def:bdtdilation}
Let  $\mathbb{L}$ be a length function on the discrete group  $\Gamma.$  Then $\Gamma$ is said to have the \textbf{property of bounded ${\bf t}$-dilation} with respect to $\mathbb{L}$ for a fixed ${\bf t}>1$ if $\mathbb{L}$ is proper and there exists $K_{\mathbb{L}} < \infty$ such that $| B_{\mathbb{L}}({\bf t}R) | \leq K_{\mathbb{L}}\ | B_{\mathbb{L}}(R)|$ for each $R \geq 1$.
\end{defn}

  A simple argument shows that $\Gamma$  has the property of bounded ${\bf t}$-dilation with respect to $\mathbb{L}$ for some ${\bf t}>1$ if and only if $\mathbb{L}$ has bounded doubling.  

Suppose $\mathbb{L}$ is a length function on a discrete group $\Gamma$.  Let $M_{\mathbb{L}}$ denote the (usually unbounded) operator on $\ell^2(\Gamma)$ of pointwise multiplication by $\mathbb{L}$, $\sigma$ a multiplier on $\Gamma$, and the left-$\sigma$ regular representation of $\ell^1(\Gamma,\sigma)$ on $\ell^2(\Gamma)$.  In $\cite{Long-Wu2}$, Long and Wu show that given any $\gamma \in \Gamma$,
$$\Vert \, [D_{\mathbb{L}}, \lambda_{\sigma}(\delta_{\gamma})] \, \Vert_{\mathcal{B}(\ell^2(\Gamma))} = \mathbb{L}(\gamma).$$
In particular, when $f \in C_C(\Gamma, \sigma),\, [D_{\mathbb{L}}, \lambda_{\sigma}(f)]$ extends to a bounded operator on $\ell^2(\Gamma);$ we note for later the following key inequality:
\begin{equation}
\label{inequal:normcommutator}
\| \, [D_{\mathbb{L}}, \lambda_{\sigma}(f)] \, \|\;\leq\;\sum_{\gamma\in\Gamma}\mathbb{L}(\gamma)|f(\gamma)|,\;f\;\in\; C_C(\Gamma, \sigma).
\end{equation}

  As $C_C(\Gamma, \sigma)$ is a dense subalgebra of $C^*(\Gamma, \sigma)$, any proper length function in this setting can be used to build a candidate Dirac operator for a spectral triple for $C^*(\Gamma, \sigma)$ that satisfies conditions (ST1) and  (ST2), but to guarantee that this spectral triple is finitely summable, one needs further conditions on $\mathbb{L}$.    In this section, we will define length functions that exhibit the additional specifications allowing the related Dirac operators to give finitely summable spectral triples, as well as satisfy the conditions of Definition \ref{def:spectraltriple}, and thus construct Dirac operators that also equip noncommutative solenoids with a quantum compact  metric space structure.

To obtain a length function on $\mathbb{Z} \Big[\frac{1}{p} \Big]$ that will produce the Dirac operator with the desired properties, we will begin by recalling the  embedding $ \Lambda: \mathbb{Z} \Big[\frac{1}{p} \Big]\to \mathbb{R} \times \mathbb{Q}_p$ (where $\mathbb{Q}_p$ are the $p$-adic rationals) as a lattice.    Also recall that the topology on $\mathbb Q_p$ is induced by the $p$-adic norm on $\mathbb Q$ defined by 
$$\|r\|_p\;= \begin{cases} p^{-n} & \text{ if $r \neq 0 $ and $r=\frac{ap^n}{b},\;a,\;n\in\mathbb Z,\;b\in \mathbb Z\backslash\{0\},\;\text{gcd}(a,p)=\text{gcd}(b,p)=\text{gcd}(a,b)=1,$ }\\ 0 & \text{ if $r = 0$.}
\end{cases}.$$

In particular, $\mathbb{Q}_p$ can be viewed as the completion of $\mathbb Q$ for the $p$-adic metric, which is the metric induced by this norm.  Since addition and multiplication are uniformly continuous for this metric, $\mathbb{Q}_p$ can also be given the structure of a field.  For further details on $p$-adic arithmetic and analysis, see the book by A. Robert \cite{robert}. Thus $\mathbb Z\Big[\frac{1}{p}\Big],$ being a subset of $\mathbb Q,$ embeds naturally into $\mathbb R$ and also into the $p$-adic field $\mathbb{Q}_p.$ To avoid confusion, we let $\iota:\mathbb Z\Big[\frac{1}{p}\Big]\;\to\;\mathbb{Q}_p$ denote the embedding of $\mathbb Z\Big[\frac{1}{p}\Big]$ into  $\mathbb{Q}_p,$ and if $r\in\;\mathbb Z\Big[\frac{1}{p}\Big]\subset \mathbb{R}$, we view it as being an element of $\mathbb R.$ We then define $\Lambda:\mathbb Z\Big[\frac{1}{p}\Big]\;\to\;\mathbb R\times \mathbb{Q}_p$ by 
$$\Lambda(r)\;=\; (r,-\iota(r)) \in \mathbb{R} \times \mathbb{Q}_p,\;r\in\;\mathbb Z\Big[\frac{1}{p}\Big].$$
Following a suggestion of J. Kaminker and J. Spielberg, Latr\'emoli\`ere and Packer showed in \cite{LaPa1} that $\Lambda$ embeds $\mathbb{Z} \Big[\frac{1}{p}\Big]$ as a lattice in $\mathbb R\times \mathbb Q_p$ with, in fact, the Pontryagin dual $\widehat{\mathbb{Z} \Big[\frac{1}{p}\Big]}\;=\;{\mathcal S}_p$ as its quotient. 
We note that $\Lambda$ was used to construct different classes of finitely generated projective modules for the $C^*$-algebras ${\mathcal A}_{\theta}^{\mathcal S}$ in \cite{LaPa1}.   In \cite{Enstad}, U. Enstad used 
the embedding $\Lambda$ to give a new example of a non-trivial bundle motivating a Balian--Low theorem in a novel context.  In \cite{Lu}, S. Lu used $\Lambda$ and some of its generalizations to obtain some new sufficient conditions for two noncommutative solenoids to be strongly Morita equivalent to one another.  This gives yet another example where the diagonal embedding of $\mathbb{Z} \Big[\frac{1}{p}\Big]$ as a lattice in $\mathbb R\times \mathbb Q_p$ is the source of novel tools.

 As a means of obtaining a proper length function on $\mathbb{Z} \Big[\frac{1}{p}\Big] \times \mathbb{Z} \Big[\frac{1}{p}\Big]$ that is of bounded doubling, we will first consider an initial length function defined on  the abelian group $\mathbb{Z} \Big[\frac{1}{p}\Big].$  

\begin{lemma}\label{lemma:baselengthfcn}
 Fix a prime $p$. Consider the function $\mathbb{L}_{p}: \mathbb{Z} \Big[\frac{1}{p}\Big] \to [0, \infty)$ given by 
 $$\mathbb{L}_p(r)\;=\;|r|+\|r\|_p.$$ Then $\mathbb{L}_p$ is a length function on $\mathbb Z\Big[\frac{1}{p}\Big]$ that is both proper and unbounded. 
\end{lemma}
\begin{proof}
	The basic requirements for a length function on a group will first be checked.  Suppose that $r, r' \in \mathbb{Z} \Big[\frac{1}{p}\Big]$.  Since $0\leq |r| \leq \mathbb{L}_p(r)$, $\mathbb{L}_p(r)=0$ if and only if $|r|=0$.  As $| \cdot |$ is a norm on $\mathbb{Z} \Big[\frac{1}{p}\Big]$, $\mathbb{L}_p$ vanishes only on the identity of $\mathbb{Z} \Big[\frac{1}{p}\Big]$. It is easily seen that $\mathbb{L}_p(r)=\mathbb{L}_p(-r),\;\forall r\in \mathbb{Z} \Big[\frac{1}{p}\Big].$ Moreover,  
	$$\mathbb{L}_p(r + r')\;=\;|r+r'|+\|r+r'\|_p\;\leq\;|r|+|r'|+\|r\|_p+\|r'\|_p\;=\mathbb{L}_p(r)+\mathbb{L}_p(r'),$$
	thereby completing the verification.

    To show that $\mathbb{L}_p$ is also proper, fix a choice of $R >0$.  
    We will demonstrate  that 
    $B_{\mathbb{L}_p}(R)$ is a finite set by making some very crude estimates on $|B_{{\mathbb L}_p}(p^d)|$, where $d$ is the least nonnegative integer such that $R < p^d$. 
    To do so  we will first  prove that, for all integers $d\geq 1$:
    \begin{equation}\label{eq:inclusions-set}\{\frac{m}{p^{d-1}}: m\in \mathbb Z,\;-p^{2(d-1)}\;\leq\;m\;\leq\;p^{2(d-1)}\}\;\subseteq B_{\mathbb{L}_p}(p^d)\;\subseteq \;\{\frac{m}{p^{d-1}}: m\in \mathbb Z,\;-p^{2d}\;\leq\;m\;\leq\;p^{2d}\},\end{equation}
    which via a  counting argument readily implies:
    \begin{equation}
    \label{eq:inclusions}
    2p^{2(d-1)}+1\;\leq\;|B_{\mathbb{L}_p}(p^d)|\;\leq\;2p^{2d}+1.
    \end{equation}
    To prove Equation \eqref{eq:inclusions-set}, let
    $x=\frac{m}{p^{2(d-1)}}$, with  $|m|\;\leq\; p^{2(d-1)}.$ Then
    $\mathbb{L}_p(\frac{m}{p^{d-1}})\leq {p^{d-1}}+{p^{d-1}}= 2p^{d-1}\leq p^d,$
    which shows the set inclusion on the left  hand side.
    
    To show the set inclusion on the right  hand side, we will actually use complements, and show that if 
    $x\in  \mathbb{Z} \Big[ \frac{1}{p} \Big]$ is not in $\{\frac{m}{p^{d-1}}: m\in \mathbb Z,\;-p^{2d}\;\leq\;m\;\leq\;p^{2d}\}$, then it is not in $B_{\mathbb{L}_p}(p^d)$.

    Fix any  $x\not\in \{ \frac{m}{p^{d-1}}: m\in \mathbb Z,\;-p^{2d}\;\leq\;m\;\leq\;p^{2d} \}$.  We first assume that we have  $x=\frac{m}{p^{d-1}}$ with $|m|>p^{2d}$ and  $\text{gcd}(m,p)=1.$  In this case, we have $\mathbb{L}_p(\frac{m}{p^{d-1}})=\;\frac{|m|}{p^{d-1}}+p^{d-1} >\frac{p^{2d}}{p^{d-1}}+p^{d-1}> \frac{p^{2d}}{p^d}= p^d.$ For the other case, suppose $x=\frac{m}{p^j}$  with $\text{gcd}(m,p)=1$ and $j\geq  d,$ then $\mathbb{L}_p(\frac{m}{p^j}) > p^j\;\geq p^d,$ so that   $x\not\in B_{{\mathbb L}_p}(p^d).$
     This  completes the proof of Equation \eqref{eq:inclusions-set}.
     
     Now notice that, if $d$ is the least nonnegative integer such that $R < p^d$:
     
     \[
    B_{\mathbb{L}_p}(R) \subseteq   B_{\mathbb{L}_p}(p^d). 
     \]

     An application of Equation \eqref{eq:inclusions} ends the proof that  $\mathbb{L}_p$ is a proper length function on $\mathbb{Z} \Big[\frac{1}{p}\Big]$.
    
    Finally, because $\mathbb{L}_p(r)>|r|,\;\forall r\in \mathbb{Z} \Big[\frac{1}{p}\Big],$ and the absolute value function is unbounded on $\mathbb{Z} \Big[\frac{1}{p}\Big],$ we see that $\mathbb{L}_p$ is unbounded.
    
\end{proof}

\begin{rmk}
\label{rem:cqms}
When a length function $\mathbb{L}$ on a discrete group is also of bounded doubling, Christ and Rieffel showed that $M_{\mathbb{L}}$ induces a quantum compact metric space structure on the group $C^*$-algebra \cite{Christ-Rieffel}.  These results were extended by Long and Wu to the setting of group $C^*$-algebras twisted by a 2-cocycle \cite{Long-Wu2}.  Furthermore, Christ and Rieffel demonstrated that such length functions also exhibit a property called \textit{polynomial growth} \cite{Christ-Rieffel}.  Connes proved that for a finitely generated discrete group $\Gamma$  with length function $\mathbb{L},$  if there exists a positive constant $s$ such that $(\text{Id} + [D_{\mathbb{L}}]^2)^{\frac{-s}{2}}$ is a trace-class operator, then $\mathbb{L}$ will have polynomial growth in which case,  $M_{\mathbb{L}}$ is also the Dirac operator for a spectral triple.  Under these hypotheses, the associated spectral triple is said to be \textit{finitely summable}. Conversely, if $\Gamma$ is a finitely generated discrete group with length function $\mathbb{L}$ such that $\mathbb{L}$ has polynomial growth, then Connes proved that the associated spectral triple is finitely summable \cite{Con89}.  In our case, the groups we consider are not finitely generated. However, as a consequence of the following lemma, our Dirac operators arise from length functions on $\mathbb{Z}\Big[\frac{1}{p}\Big] \times \mathbb{Z}\Big[\frac{1}{p}\Big]$ that can also be shown to be of bounded doubling, and therefore by Proposition 1.2 of \cite{Christ-Rieffel}, the length functions that we construct will have polynomial growth.   Moreover, we shall show in Theorem \ref{thm:ncsolenoidspectraltriple} that bounded doubling of the length function, even in the non-finitely generated case, implies that the associated spectral triple is finitely summable.
\end{rmk}

\begin{lemma}\label{lemma:boundeddoubling}
	For a fixed prime $p$, $\mathbb Z\Big[\frac{1}{p}\Big]$ has the property of bounded $p$-dilation with respect to $\mathbb{L}_p$.  In particular, $\mathbb{L}_p$ is of bounded doubling with $C_{\mathbb{L}_p} = 4 p^8$.
\end{lemma}

\begin{proof}
  As a consequence of the proof of  Lemma \ref{lemma:baselengthfcn}, particularly Equation \eqref{eq:inclusions}, we have  
\begin{align}|B_{\mathbb{L}_p}(p^{d+1})|&\leq |B_{\mathbb{L}_p}(p^{d+1})| +|B_{\mathbb{L}_p}(p^d)|\;\leq\;(2p^{2(d+1)}+1) +(2p^{2d}+1) < 2p^{2(d+1)} + 2p^{2d} + 2 p^4 \label{eq:ineq}\notag \\ 
&= p^4(2p^{2(d-1)}) + p^2(2p^{2(d-1)}) + 2p^4 < 2p^4(2p^{2(d-1)}+1) \leq 2p^4(|B_{\mathbb{L}_p}(p^d)|).
\end{align}
When $d$ is zero, note that $\mathbb{L}_p(0) = 1$ implies $|B_{\mathbb{L}_p}(1)| \geq 1$, which yields
$$|B_{\mathbb{L}_p}(p)|\leq 2p^2+1 < 2 p^4 \leq 2p^4|B_{\mathbb{L}_p}(1)|.$$
In particular, given $R\geq 1,$ there exists $d\in\mathbb N$ such that $p^d\leq R<p^{d+1},$ and so 
\begin{equation}
\label{eq:ballcontainment}
B_{\mathbb{L}_p}(p^{d})\;\subseteq\; B_{\mathbb{L}_p}(R)\subseteq\; B_{\mathbb{L}_p}(pR)\;\subseteq\;B_{\mathbb{L}_p}(p^{d+2}),
\end{equation}
which yields, together with Equation \eqref{eq:ineq}:
\begin{equation}
\label{eq:anotherkeydoublingineq}
|B_{\mathbb{L}_p}(pR)|\;\leq\;|B_{\mathbb{L}_p}(p^{d+2})|\;\leq\;4p^8|B_{\mathbb{L}_p}(p^d)|\;\leq\;4p^8|B_{\mathbb{L}_p}(R)|.   \end{equation}
Since the same reasoning holds for each $R \geq 1$, $\mathbb Z\Big[\frac{1}{p}\Big]$ has the property of bounded $p$-dilation with respect to $\mathbb{L}_p$.  The conclusion that $\mathbb{L}_p$ is also of bounded doubling follows by our remark made after Definition \ref{def:bdtdilation}.
\end{proof}	

Latr\'emoli\`ere and Packer used the map $\Lambda$ to build an embedding of $\mathbb{Z}\Big[\frac{1}{p}\Big]\times \mathbb{Z} \Big[\frac{1}{p}\Big]$ as a lattice in $[\mathbb{R} \times \mathbb{Q}_p] \times [\mathbb{R} \times \mathbb{Q}_p]$ \cite{LaPa1}.  As a consequence, $\mathbb{L}_p$ can be used to build a length function on $\mathbb{Z}\Big[\frac{1}{p}\Big]\times \mathbb{Z} \Big[\frac{1}{p}\Big]$ that is of bounded doubling.  More generally,

\begin{prop}\label{prop:lengthfcn sum}
    Let $\Gamma_1$ and $\Gamma_2$ be countable discrete  groups and $\Gamma=\Gamma_1 \times \Gamma_2$.  Given proper length functions $\mathbb{L}^{(i)}$ on $\Gamma_i$ for $i = \{ 1, 2 \}$, define $\mathbb{L}^{\Sigma}:\Gamma\to [0,\infty)$ by, for all   $\;x_1 \in \Gamma_1,\ x_2\ \in \Gamma_2$:
	\begin{equation}\label{eq:length func}
	\mathbb{L}^{\Sigma}((x_1,x_2))\; :=\;\mathbb{L}^{(1)}(x_1)+\mathbb{L}^{(2)}(x_2).   
	\end{equation}

	Then $\mathbb{L}^{\Sigma}$ is a proper length function on $\Gamma.$
If $\mathbb{L}^{(i)}$ has the bounded doubling property with $C_{\mathbb{L}^{(i)}}$ for $i \in \{1, 2 \},$   then $\mathbb{L}^{\Sigma}$ is of bounded doubling with $C_{\mathbb{L}^{\Sigma}} = C^4,$ where $C = \max \{ C_{\mathbb{L}^{(1)}}, C_{\mathbb{L}^{(2)}} \}.$
\end{prop}

\begin{proof}
A length function on a Cartesian product of two groups can be built from a sum of length functions on each group.  Since length functions take only non-negative values, $\mathbb{L}^{\Sigma}$ vanishes if and only if $\mathbb{L}^{(1)}$ and $\mathbb{L}^{(2)}$ both vanish.  Because $\mathbb{L}^{(1)}$ and $\mathbb{L}^{(2)}$ each evaluate to zero only on $e_{\Gamma_1}$ and $e_{\Gamma_2}$ respectively, $\mathbb{L}^{\Sigma}$ likewise takes null values only on $e_{\Gamma}$.  Next suppose that $(x_1, x_2), (x_1', x_2') \in \Gamma$.  The remaining length function properties of $\mathbb{L}^{(1)}$ and $\mathbb{L}^{(2)}$ imply

$$\mathbb{L}^{\Sigma}((x_1,x_2)^{-1}) = \mathbb{L}^{\Sigma}((x_1^{-1},x_2^{-1})) = \mathbb{L}^{(1)}(x_1^{-1})+\mathbb{L}^{(2)}(x_2^{-1})  = \mathbb{L}^{(1)}(x_1)+\mathbb{L}^{(2)}(x_2) = \mathbb{L}^{\Sigma}((x_1,x_2)),$$
as well as 
\[
\begin{split}
\mathbb{L}^{\Sigma}((x_1,x_2)\,(x_1',x_2')) &= \mathbb{L}^{\Sigma}((x_1\,x_1',x_2\,x_2')) = \mathbb{L}^{(1)}(x_1\, x_1')+ \mathbb{L}^{(2)}(x_2\, x_2')\\ 
&\leq \mathbb{L}^{(1)}(x_1) + \mathbb{L}^{(1)}(x_1')+ \mathbb{L}^{(2)}(x_2)+ \mathbb{L}^{(2)}(x_2') = \mathbb{L}^{\Sigma}((x_1,x_2)) + \mathbb{L}^{\Sigma}((x_1',x_2')).
\end{split}
\]  
Thus $\mathbb{L}^{\Sigma}$ is also a length function on $\Gamma.$ 

Fix an $0<R < \infty$.  If $\mathbb{L}^{\Sigma}((x_1, x_2)) \leq R$, then this same upper bound holds both for $\mathbb{L}^{(1)}(x_1)$ and for $\mathbb{L}^{(2)}(x_2)$.  For each $i \in \{ 1, 2 \}$, define $B_{\mathbb{L}^{(i)}}(R)$ and $B_{\mathbb{L}^{\Sigma}}(R)$ as in part (1) of Definition \ref{def:lengthfcnbddoubl}.  Since $\mathbb{L}^{(1)}$ and $\mathbb{L}^{(2)}$ are also each proper length functions, $B_{\mathbb{L}^{(1)}}(R)$ and $B_{\mathbb{L}^{(2)}}(R)$ are finite sets, as is $B_{\mathbb{L}^{\Sigma}}(R) \subseteq B_{\mathbb{L}^{(1)}}(R) \times B_{\mathbb{L}^{(2)}}(R)$.  Because the same reasoning holds for each $R < \infty$, $\mathbb{L}^{\Sigma}$ is a proper length function on $\Gamma$.  Next, assume that ${\mathbb L}_i$ satisfies the bounded doubling property with constants $C_{{\mathbb L}_i}$ for $i\in \{1,2\},$ and choose an $R \geq 1$.  Repeated application of the bounded doubling condition yields 
\begin{equation}
\label{eq:bddoublesum}
| \, B_{\mathbb{L}^{(1)}}(2R) \, | \, \cdot \, | \, B_{\mathbb{L}^{(2)}}(2R) \, | \leq C^4 \, | \, B_{\mathbb{L}^{(1)}}(R/2)\, | \, \cdot \, | \, B_{\mathbb{L}^{(2)}}(R/2) \, |.
\end{equation}
Note that $B_{\mathbb{L}^{\Sigma}}(2R) \subseteq B_{\mathbb{L}^{(1)}}(2R) \times B_{\mathbb{L}^{(2)}}(2R)$ and $B_{\mathbb{L}^{(1)}}(R/2) \times B_{\mathbb{L}^{(2)}}(R/2) \subseteq B_{\mathbb{L}^{\Sigma}}(R)$, hence $| \, B_{\mathbb{L}^{\Sigma}}(2R) \, |$ is bounded above by $C^4 | \, B_{\mathbb{L}^{\Sigma}}(R) \, |.$  As the choice of $R \geq 1$ was arbitrary, $\mathbb{L}^{\Sigma}$ is of bounded doubling with $C_{\mathbb{L}^{\Sigma}} = C^4$.
\end{proof}
\begin{cor}\label{cor:length function sum}
Fix a prime $p$, and   let $\mathbb{L}_p^{\Sigma}: \mathbb{Z}\Big[\frac{1}{p} \Big] \times \mathbb{Z}\Big[\frac{1}{p}\Big] \to [0, \infty)$ be given by
$$\mathbb{L}_p^{\Sigma}(\gamma_1, \gamma_2): = \mathbb{L}_p(\gamma_1) + \mathbb{L}_p(\gamma_2).$$
Then $\mathbb{L}_p^{\Sigma}$ is a length function on $\mathbb{Z}\Big[\frac{1}{p} \Big] \times \mathbb{Z}\Big[\frac{1}{p}\Big]$ that is of bounded doubling with $C_{\mathbb{L}_p^{\Sigma}}=(4p^8)^4.$
\end{cor}

\subsection{Spectral Triples on Noncommutative Solenoids from Length Functions with Bounded Doubling}
\label{subsec:SpecTripSolandindlims}

Throughout this section, fix a prime $p$ and let $\Gamma=\mathbb{Z}\Big[\frac{1}{p}\Big]\times \mathbb{Z}\Big[\frac{1}{p}\Big].$  Denote $\ell^2(\Gamma)$ by $H.$
To build our Dirac operators, we recall from Definition \ref{def:ncsolenoid} that the noncommutative solenoid $\mathcal{A}_{\theta}^{\mathcal S}$ is defined to be the twisted group $C^*$-algebra $C^*\Big(\Gamma, \sigma_{\theta}\Big)$. 
 Moreover,  $C^*\Big(\Gamma, \sigma_{\theta}\Big)$ has a faithful representation as bounded operators on the Hilbert space  $H$ via 
the left-$\sigma$ regular representation $\lambda_{\sigma_{\theta}}$.   Let $\mathcal{D}_{p}$ be the unbounded operator $M_{\length}$ defined on a dense subspace of $H.$  Since $\length$ is  proper, $\mathcal{D}_{p}$ can also be shown to satisfy conditions (ST1) and (ST2) of Definition \ref{def:spectraltriple} for a Dirac operator. Moreover, we will show that the spectral triple coming from $\mathcal{D}_{p}$ is finitely summable. That is, we have:
\begin{thm}\label{thm:ncsolenoidspectraltriple}
Fix a prime $p$.  Let $\Gamma = \mathbb{Z}\Big[\frac{1}{p}\Big]\times \mathbb{Z}\Big[\frac{1}{p}\Big]$ and let $H=\ell^2(\Gamma).$  For every $\theta \in \Omega_p,\; ({\mathcal A}_{\theta}^{\mathcal S}, H, \mathcal{D}_{p})$ with representation $\lambda_{\sigma}$ is a spectral triple for the noncommutative solenoid ${\mathcal A}_{\theta}^{\mathcal S}$ with associated smooth subalgebra $C_C(\Gamma,\sigma_{\theta}).$ Furthermore, this spectral triple is finitely summable.
\end{thm}
\begin{proof}
The following is a slight modification of the proof of Long and Wu's Proposition 3.3 in \cite{Long-Wu2}. 	
	
For every $\gamma \in \Gamma$, $\delta_{\gamma}$ is an eigenvector of $\mathcal{D}_{p}$ with eigenvalue $\mathbb{L}_p^{\Sigma}(\gamma)$.  Hence each such $\delta_{\gamma}$ is also an eigenvector of $(\mathcal{D}_{p}- it \text{Id}_H)^{-1}$, $t \in \mathbb{R} \setminus \{ 0 \}$, with eigenvalue $(\mathbb{L}_p^{\Sigma}(\gamma) - it)^{-1}$.  To show that $(\mathcal{D}_{p}- it\text{Id}_H)^{-1}$ is compact for a fixed $t \in \mathbb{R} \setminus \{ 0 \}$, choose an $\epsilon >0$.  Then there exists $R>0$ such that for $z \in \mathbb{R}$, $z > R$ implies $|z - it|^{-1} < \epsilon$.  Given such an $R$, define $B_{\mathbb{L}_p^{\Sigma}}(R)$ as in  part (1) of Definition \ref{def:lengthfcnbddoubl}.  Since $\mathbb{L}_p^{\Sigma}$ is proper, $B_{\mathbb{L}_p^{\Sigma}}(R)$ is a finite set.  Let $\{ \gamma_j \}_{1 \leq j \leq m}$ be an enumeration of the elements in this set.  For each $1 \leq j \leq m$, let $P_j$ signify the orthogonal projection taking $H$ onto the one-dimensional subspace spanned by $\delta_{\gamma_j}$.  Then 
$$\| (\mathcal{D}_{p} - it\text{Id}_H)^{-1} - \Sigma_{j=1}^m (\mathbb{L}_p^{\Sigma}(\gamma_j) - it)^{-1} P_j \|_{\mathcal{B}(H)} = \sup \left\{ |\mathbb{L}_p^{\Sigma}(\gamma) - it|^{-1} : \gamma \notin B_{\mathbb{L}_p^{\Sigma}}(R)  \right\} < \epsilon.$$
As the choice of $\epsilon$ was arbitrary, $(\mathcal{D}_{p} - it\text{Id}_H)^{-1}$ can be approximated  in norm by finite rank operators.  Hence $\mathcal{D}_{p}$ satisfies condition (ST1).  Since $\mathbb{L}_p^{\Sigma}$ is an unbounded, proper length function as a consequence of Lemma \ref{lemma:baselengthfcn} and Proposition \ref{prop:lengthfcn sum}, the associated Dirac operator $\mathcal{D}_{p}$ and subalgebra $C_C(\Gamma,\sigma_{\theta})$ also satisfy condition (ST2), by Proposition 3.3 of \cite{Long-Wu2}.

We finally prove that the spectral triple $({\mathcal A}_{\theta}^{\mathcal S}, H, \mathcal{D}_{p})$ is finitely summable, i.e. we prove that there exists $s>0$ such that for all $t>s,\;(\text{Id}_H + [\mathcal{D}_{p}]^2)^{\frac{-t}{2}}$ is a trace class operator.  For this we use the bounded doubling condition proved for $\mathcal{D}_{p}$ in Corollary \ref{cor:length function sum}.  Recall there exists a constant $C_{\mathbb{L}_p^{\Sigma}}>1$ such that for every $R\;\geq\;1,\;|B_{\mathbb{L}_p^{\Sigma}}(2R)|\leq C_{\mathbb{L}_p^{\Sigma}}|B_{\mathbb{L}_p^{\Sigma}}(R)|.$  Taking $R=1,$ by induction, we obtain for every $n\in\mathbb N$ that 
$$|B_{\mathbb{L}_p^{\Sigma}}(2^n)|\leq\;C_{\mathbb{L}_p^{\Sigma}}|B_{\mathbb{L}_p^{\Sigma}}(2^{n-1})|\;\text{and}\;|B_{\mathbb{L}_p^{\Sigma}}(2^n)|\leq\;(C_{\mathbb{L}_p^{\Sigma}})^n|B_{\mathbb{L}_p^{\Sigma}}(1)|.$$
Fix $t>0.$ The eigenfunctions for $(\text{Id}_H + [\mathcal{D}_{p}]^2)^{\frac{-t}{2}}$ are exactly the delta functions $\{\delta_{\gamma}\;:\;\gamma\in \Gamma\}$ with corresponding eigenvalues equal to 
$\Big\{\frac{1}{(1+(\mathbb{L}_p^{\Sigma}(\gamma))^2)^{t/2}}\;:\;\gamma\in \Gamma\Big\}.$ Let $\delta_{\gamma_0}$ be a fixed eigenfunction for $\mathcal{D}_{p}.$
   To obtain an upper bound on the modulus of the eigenvalue corresponding to $\delta_{\gamma_0},$ we need to determine where $\gamma_0$ sits in the partition
$$\Gamma\;=\;B_{\mathbb{L}_p^{\Sigma}}(1)\;\bigsqcup\;\sqcup_{n=1}^{\infty}[B_{\mathbb{L}_p^{\Sigma}}(2^n)\backslash B_{\mathbb{L}_p^{\Sigma}}(2^{n-1})].$$
  To check that  $(\text{Id}_H + [\mathcal{D}_{p}]^2)^{\frac{-t}{2}}$  is trace class, we just need to check that the sum of moduli of the eigenvalues of this operator, counted with multiplicity, converge.
We note that within the ball $B_{\mathbb{L}_p^{\Sigma}}(1),\;(\text{Id}_H + [\mathcal{D}_{p}]^2)^{\frac{-t}{2}}$ has $|B_{\mathbb{L}_p^{\Sigma}}(1)|$ eigenfunctions having eigenvalues with modulus less than or equal to 
$\frac{1}{(1+(0)^2)^{t/2}}=1.$
Then for $n\geq 1,$ each annulus $[B_{\mathbb{L}_p^{\Sigma}}(2^n)\backslash B_{\mathbb{L}_p^{\Sigma}}(2^{n-1})]$
has cardinality less than or equal to 
$$C_{\mathbb{L}_p^{\Sigma}}|B_{\mathbb{L}_p^{\Sigma}}(2^{n-1})|-|B_{\mathbb{L}_p^{\Sigma}}(2^{n-1})|\;\leq\;(C_{\mathbb{L}_p^{\Sigma}}-1)(C_{\mathbb{L}_p^{\Sigma}})^{n-1}|B_{\mathbb{L}_p^{\Sigma}}(1)|.$$
For each eigenfunction $\delta_{\gamma_0}$ corresponding to $\gamma_0$ lying in this $n^{\text{th}}$ annulus, the corresponding eigenvalues have modulus less than $\Big(\frac{1}{(1+2^{2(n-1)})}\Big)^{t/2}.$
These estimates show that the trace of 
 $(\text{Id}_H + [\mathcal{D}_{p}]^2)^{\frac{-t}{2}}$  is bounded from above by:
\[
\begin{split}
|B_{\mathbb{L}_p^{\Sigma}}(1)|\cdot &1+\sum_{n=1}^{\infty}(C_{\mathbb{L}_p^{\Sigma}}-1)(C_{\mathbb{L}_p^{\Sigma}})^{n-1}|B_{\mathbb{L}_p^{\Sigma}}(1)|\cdot \frac{1}{(1+2^{2(n-1)})^{t/2}}\\
&\leq\;\text{max}\{1, (C_{\mathbb{L}_p^{\Sigma}}-1)\}\cdot|B_{\mathbb{L}_p^{\Sigma}}(1)|\cdot \Big[1+\sum_{n=1}^{\infty}\frac{(C_{\mathbb{L}_p^{\Sigma}})^{n-1}}{(1+2^{2(n-1)})^{t/2}}\Big]\\
&\leq\;\text{max}\{1, (C_{\mathbb{L}_p^{\Sigma}}-1)\}\cdot|B_{\mathbb{L}_p^{\Sigma}}(1)|\cdot \Big[1+ \sum_{n=1}^{\infty} \frac{(C_{\mathbb{L}_p^{\Sigma}})^{n-1}}{(2^t)^{n-1}}\Big]\\
&=\;\text{max}\{1, (C_{\mathbb{L}_p^{\Sigma}}-1)\}\cdot|B_{\mathbb{L}_p^{\Sigma}}(1)|\cdot \Big[1+\sum_{n=1}^{\infty}\Big[\frac{C_{\mathbb{L}_p^{\Sigma}}}{2^t}\Big]^{n-1}\Big],\\
\end{split}
\]
and this last series converges as long as $2^{t}>C_{\mathbb{L}_p^{\Sigma}},$ that is, it converges for all $t>\log_2(C_{\mathbb{L}_p^{\Sigma}}).$
Therefore, the spectral triple $({\mathcal A}_{\theta}^{\mathcal S}, H, \mathcal{D}_{p})$ is finitely summable.
\end{proof}

\subsection{Noncommutative Solenoids as Quantum Compact  Metric Spaces}
\label{subsec:SolCQMS}

In the context of spectral triples  and other more general settings, Rieffel developed a theory (later expanded by Latr\'emoli\`ere) that can be termed \textit{noncommutative metric geometry}, see Rieffel \cite{Rieffel, Rieffel2}.  The basis for this perspective is built on a noncommutative analogue of a compact metric space.  We review some of this work below.  Whereas Rieffel gave the definition below for a seminorm defined on an order--unit space taking finite values, Latr\'emoli\`ere  restricted  to the unital $C^*$-algebra setting, where the seminorm in question takes on finite values on a dense subalgebra of a unital $C^*$-algebra.  Because of the extra properties required, he decided to switch the order of ``compact" and ``quantum".  
  We will give below an adaptation  of the slightly more specialized definition given by Latr\'emoli\`ere, see \cite{La}. Cf. \cite[Remark 2.7]{La}  for a precise comparison of 
Rieffel and Latr\'emoli\`ere's definitions.

\begin{defn}	\label{def:qcms}  
	A {\bf {\qcms} $(A,L)$} is an ordered pair where $A$ is a unital C*-algebra and $L$ is a seminorm defined on a dense $\ast$-subalgebra $\dom{L}$ of the self-adjoint elements $sa(A)$ such that:
	\begin{enumerate}
		\item[(1)] $\{ a \in sa(A) : L(a) = 0 \} = \R\unit_A$,
		\item[(2)] the {\MongeKant} $\Kantorovich{L}$, defined on the state space ${\mathcal S}(A)$ of $A$ by setting for all $\varphi,\psi \in \mathcal{S}(A)$:
		\begin{equation*}
		\Kantorovich{L}(\varphi,\psi) = \sup\left\{ \left|\varphi(a)-\psi(a)\right| : a\in\dom{L}, L(a)\leq 1 \right\}
		\end{equation*}
		metrizes the weak* topology restricted to the state space $\mathcal{S}(A)$ of $A$.
	\end{enumerate}
	Moreover, seminorms $L$ on $A$ that satisfy conditions (1) and (2) are called  {\bf Lip-norms}.
\end{defn}

\begin{example}
	\label{example_qcms}(Rieffel \cite{Rieffel})
	Let $(X, d)$ be a compact metric space.  Set $L_d$ equal to the Lipschitz seminorm on $C(X)$ associated to $d$- that is, for every $f$ in $C(X)$, 
	$$L_{d}(f) = \sup \Big\{ \frac{|f(x) - f(y)|}{d(x,y)}: x,y \in X, x \neq y \Big\}.$$
	$\newline$
Then $(C(X), L_d)$ is a quantum compact metric space.  It is well-known  that the metric $d$ can be recovered from the restriction of $mk_{L_d}$ to the pure states.
\end{example}

\begin{rmk}  In Proposition  1.3 of \cite{Rieffel-Ozawa}, it is shown that for a unital  $C^*$-algebra $A$ with a trace $\tau,$ a Lipschitz seminorm $L$ for $A$ will be a Lip-norm on $A$ if and only if $\{ a \in \text{dom}(L) : L(a) \leq 1,\;\tau(a)=0 \}$ is totally bounded for the topology induced by $\| \cdot \|_A$  on $\text{dom}(L)$.
	
	  There is an additional property that can be assigned to Lipschitz seminorms, which is the {\bf Leibniz property}:  a Lipschitz seminorm $L$ on $A$ is said to satisfy the {\bf Leibniz property} if, for all $a, b \;\in\; \text{dom}(L)$: 
	  $$L(ab)\;\leq\;\|a\|_A\cdot L(b)+\|b\|_A\cdot L(a).$$
	  
	  In his studies and generalizations of compact quantum metric spaces, Latr\'emoli\`ere in \cite{La} introduces an additional Leibniz condition for Lip-norms of Jordan-Lie type: that is, for all  $a, b \in \text{dom}(L),$  
 $$\frac{ab + ba}{2}, \frac{ab - ba}{2i} \in \text{dom}(L),$$
and $L$ satisfies the {\bf Leibniz inequality of  Jordan-Lie type}:
$$	\max \Big\{ L \Big(\frac{ab + ba}{2} \Big), L \Big(\frac{ab -ba}{2i} \Big) \Big\} \leq L(a) \, \| b \|_A + \| a \|_A \, L(b).$$

Therefore we make the following definition.

\begin{defn} (\cite{La} Definition 2.19, \cite{LaPa2} Definition 1.4)
	\label{def:fredLeibqcms}
	Let $A$ be a unital $C^*$-algebra, and let $L$ be a Lip-norm on a dense subspace $\text{dom}(L)$ of $A.$  We say that the ordered pair $(A,L)$ is a {\bf Leibniz quantum compact metric space} if in addition to conditions (1) and (2) of Definition \ref{def:qcms}, the following conditions are satisfied:
	\begin{enumerate}
	\item[(3)] The seminorm $L$ satisfies the Leibniz inequality of Jordan-Lie type.
	
	\item[(4)] The seminorm $L$ is lower semicontinuous with respect to the norm $\|\cdot\|_A$ topology restricted to $\text{dom}(A).$
	\end{enumerate}
\end{defn}

In Proposition 2.17 of \cite{La},  Latr\'emoli\`ere proves that if $(A,L)$ is a quantum compact metric space in the sense of Definition \ref{def:qcms}, with $L$ satisfying the Leibniz property, then $(A,L)$ will satisfy the Leibniz inequality of  Jordan-Lie type.  By Long and Wu's Proposition 6.3 in \cite{Long-Wu2}, it is the case that the Lip-norms we discuss coming from length functions will satisfy the Leibniz  property, and therefore they satisfy the Leibniz inequality of Jordan-Lie type. We will use this fact in proving Theorem \ref{thm:ncsolenoidqcms}.
	\end{rmk}

Noncommutative metric geometry allows the theory of metric geometry  to be applied to noncommutative $C^*$-algebras via their state spaces.  As with classical compact metric spaces, the set of quantum compact metric spaces can be equipped with a metric.  In analogy with its classical counterpart,  Latr\'emoli\`ere's Gromov--Hausdorff propinquity on Leibniz quantum compact metric spaces, which builds on the earlier work of Rieffel, has many of the same properties as in the compact metric space setting \cite{La}. 

 Inductive limit $C^*$-algebras admit natural candidates for quantum metric approximations.  Aguilar and Latr\'{e}moli\`{e}re have identified some conditions under which inductive limit convergence can also be described by metric convergence \cite{Agu, AgLat}, and T. Landry, M. Lapidus, and  Latr\'{e}moli\`{e}re in \cite{Lan-Lap-Lat} studied metric approximation of spectral triples on fractals of Sierpinski gasket type .  Latr\'{e}moli\`{e}re and Packer also considered the question of metric convergence in the setting of noncommutative solenoids \cite{LaPa2}.  When a compact group acts ergodically on a unital $C^*$-algebra, any length function on the group induces a quantum compact metric space structure on the $C^*$-algebra \cite{Rieffel2}.  Working in the more general context of inductive limits of $C^*$-algebras on which projective limits of compact metrizable groups act ergodically, Latr\'{e}moli\`{e}re and Packer showed that noncommutative solenoids are limits, for the Gromov-Hausdorff propinquity, of rotation algebras.  Their result relied on the definition of certain length functions to construct Lip-norms on an inductive sequence of rotation algebras and their inductive limit.  These particular Lip-norms did not come from Dirac operators,  as in the earlier work of Connes,  \cite{Con89}.  Since we want to study noncommutative solenoids as both quantum compact metric spaces as well as noncommutative Riemannian manifolds, we now describe how our results of Sections  \ref{subsec:lengthfunctionsonZp} and \ref{subsec:SpecTripSolandindlims} give noncommutative solenoids the structure of quantum compact metric spaces.

In \cite{Con89}, Connes studied spectral triples on reduced group $C^*$-algebras with Dirac operators built in the same way as $\mathcal{D}_p$ from length functions on discrete groups.  Given a unital faithful representation $\pi: A \rightarrow \mathcal{B}(H)$, let $L_{D}$ be the seminorm on the associated smooth subalgebra ${\mathcal A}$ given by
\begin{equation}
\label{grouplengthseminorm}
    L_D(a) = \| \, [D, \pi(a) ] \, \|_{\mathcal{B}(H)}.
\end{equation}
As in the unital commutative case, Connes showed that a seminorm defined as in Equation (\ref{grouplengthseminorm}) from such a Dirac operator gives rise to a metric on the state space of the reduced group $C^*$-algebra \cite{Con89}, \cite{Christ-Rieffel}.  This metric is precisely the Monge-Kantorovich metric given in part (2) of Definition \ref{def:qcms}. 
When $D$ is the Dirac operator constructed by Connes from a proper length function $\mathbb{L}$ on a discrete group $\Gamma$, Christ and Rieffel showed that $\text{mk}_{L_D}$ also metrizes the weak$^*$-topology on $\mathcal{S}(C_r^*(\Gamma))$ if $\mathbb{L}$ is in addition assumed to be of bounded doubling.  In a more general setting, Rieffel, as well as Ozawa in collaboration with Rieffel, identified alternate characterizations of Lip-norms \cite{Rieffel2}, \cite{Rieffel-Ozawa}.  In \cite{Long-Wu2}, Long and Wu used one of these equivalent sets of conditions to extend Christ and Rieffel's work to discrete groups equipped with a 2-cocycle.  Since noncommutative solenoids are defined to be twisted discrete group $C^*$-algebras, our Dirac operators can also be shown to give rise to Leibniz Lip-norms on noncommutative solenoids.

\begin{thm}\label{thm:ncsolenoidqcms}
Fix a prime $p$.  For every $\theta \in \Omega_p$, let  $(\mathcal{A}_{\theta}^{\mathcal S}, H, \mathcal{D}_p)$ be the spectral triple from Theorem \ref{thm:ncsolenoidspectraltriple}.  Then  $(\mathcal{A}_{\theta}^{\mathcal S}, L_{\mathcal{D}_p})$ is a Leibniz quantum compact metric space in the sense of Definition \ref{def:fredLeibqcms}.   
\end{thm}
\begin{proof}
 Because of Corollary \ref{cor:length function sum}, the conclusion follows immediately from Propositions 6.4 and 6.9 of \cite{Long-Wu2}, which use the property of bounded doubling for the length function to derive that the seminorm $L_{\mathcal{D}_p}$ has the Leibniz property and is lower semicontinuous, respectively. 
\end{proof}
\begin{rmk}
	\label{rmk:uniformlyeqmetrics}
The definition of  compact quantum  metric space used by Long and Wu in \cite{Long-Wu2} is that used by Christ and Rieffel in \cite{Christ-Rieffel}.  Christ and Rieffel were the first to remark that additional conditions might be desired in the definition of a quantum  compact metric space, such as lower semi-continuity for the given seminorm with respect to the $C^*$-algebra norm.  Once again we note the fact that equipped with the weak$^*$-topology, the state space of a unital $C^*$-algebra is compact and second countable, so that any two metrics which induce this same topology on the compact  space must be uniformly  equivalent. The Leibniz Lip-norm on the state space of the noncommutative solenoid obtained by Latr\'emoli\`ere and Packer in \cite{LaPa2} induces the weak$^*$-topology, as does the Leibniz Lip-norm induced by the seminorm coming from the spectral triple, by virtue of the bounded doubling condition being satisfied by the length function ${\mathbb L}_p^{\Sigma}.$  Therefore the metrics on the state space given by these two seminorms must be uniformly  equivalent.
\end{rmk}

\subsection{Restriction of Spectral Triples to Noncommutative Tori}
Recall from Theorem \ref{thm:DLTorus} that noncommutative solenoids can be identified with inductive limits of rotation algebras.  We will next build spectral triples for rotation algebras via suitably defined restrictions of the spectral triples  on noncommutative solenoids constructed in Theorem \ref{thm:ncsolenoidspectraltriple}.  For each $n\in \mathbb N,$ denote the subgroup $\frac{1}{p^n} \mathbb{Z}\times \frac{1}{p^n} \mathbb{Z}$ of $\mathbb{Z}\Big[\frac{1}{p}\Big]\times\mathbb{Z}\Big[\frac{1}{p}\Big]=\Gamma$ by $\Gamma_n.$ Recall from Section 2 that for each prime $p$, the group $\Gamma$ can be realized as an inductive limit of the groups $\Big\{ \Gamma_n \Big\}_{n \in \mathbb{N}}$, where the group monomorphisms $\alpha_n: \Gamma_n \to \Gamma_{n+1}$ are given by inclusion maps. We will show that the length function on $\Gamma_n$ that is the restriction of that given for $\Gamma$ is still of bounded doubling.
\begin{lemma}\label{lemma:restriction bounded doubling}
	For fixed prime $p$ and $n\in \mathbb N$, let $\mathbb{L}_{p,n}$ denote the restriction of $\;\mathbb{L}_p$ to $\frac{1}{p^n}\mathbb{Z}$.  Then $\frac{1}{p^n}\mathbb{Z}$ has the property of bounded $p$-dilation with respect to $\mathbb{L}_{p,n}$.  In particular, $\mathbb{L}_{p,n}$ is of bounded doubling with $C_{\mathbb{L}_{p,n}} = 4 p^8$.
\end{lemma}	
\begin{proof}
In order to qualify as a length function on $\frac{1}{p^n}\mathbb{Z}$, $\mathbb{L}_{p,n}$ must be well-defined on the identity of the group, the inverses of its elements, and the closure of those elements under the group operation.  Since $\frac{1}{p^n}\mathbb{Z} \leq \mathbb{Z}\Big[\frac{1}{p}\Big]$, $\mathbb{L}_{p}$ restricts to a length function in this setting.  Since $B_{\mathbb{L}_{p,n}}(R) \subseteq B_{\mathbb{L}_{p}}(R)$ for each $0< R < \infty$, $\mathbb{L}_{p,n}$ is also a proper length function on $\frac{1}{p^n}\mathbb{Z}$.   

The argument that $\mathbb{L}_{p,n}$ has the property of bounded $p$-dilation follows the same strategy as in Lemma \ref{lemma:boundeddoubling}. Fix $R \geq 1.$  Define $B_{\mathbb{L}_{p,n}}(R)$ as in part (1) of Definition \ref{def:lengthfcnbddoubl}.  Choose $d \in \mathbb{N}$ so that Equation \eqref{eq:anotherkeydoublingineq} holds for balls defined with $\mathbb{L}_{p,n}$.  Consider first the case when $1 \leq d \leq n$.  Since $B_{\mathbb{L}_{p,n}}(pR) \subseteq B_{\mathbb{L}_{p}}(pR)$, the inequalities described in Equation \eqref{eq:anotherkeydoublingineq} can be applied to obtain the desired bounded doubling constant if $B_{\mathbb{L}_{p}}(R)$ can be shown to be a subset of $B_{\mathbb{L}_{p,n}}(R)$.  Note that every nonzero element in $B_{\mathbb{L}_{p}}(R)$ can be written as $\frac{a}{p^k}$, where $a$ is an integer coprime with $p$ and $k \in \mathbb{N}$.  Since $\mathbb{L}_p\Big(\frac{a}{p^k}\Big) > p^k$, $\frac{a}{p^k} \in B_{\mathbb{L}_{p}}(R)$ implies $k < d$.  Because $d \leq n$, $\mathbb{L}_p\Big(\frac{a}{p^k}\Big)$ coincides with $\mathbb{L}_{p,n}\Big(\frac{a}{p^k}\Big)$.  As a result,
$$|B_{\mathbb{L}_{p,n}}(pR)| \leq |B_{\mathbb{L}_p}(pR)|\;\leq\;4p^8|B_{\mathbb{L}_p}(R)| \leq 4p^8|B_{\mathbb{L}_{p,n}}(R)|.$$
For the case when $d > n$, note that $a \in \mathbb{Z}$ and $|a| \geq p^{d+n}$ implies $\mathbb{L}_{p,n}\Big( \frac{a}{p^n} \Big) > \Big| \frac{a}{p^n} \Big| \geq p^d$.  Moreover, $a \in \mathbb{Z}$ and $|a| \leq p^{d+n-1}$, as well as $n < d$, gives 
$$\mathbb{L}_{p,n}\Big( \frac{a}{p^n} \Big) = \Big| \frac{a}{p^n} \Big| + \Big|\Big| \frac{a}{p^n} \Big|\Big|_{p} \leq p^{d-1}
+ p^n \leq 2p^{d-1} \leq p^d.$$
Therefore, 
$$\Big\{ \frac{a}{p^{n}}: a \in \mathbb{Z}, -p^{d+n-1} \leq a \leq p^{d+n-1} \Big\} \subseteq B_{\mathbb{L}_{p,n}}(p^d) \subseteq \Big\{ \frac{a}{p^{n}}: a \in \mathbb{Z}, -p^{d+n} \leq a \leq p^{d+n} \Big\},$$
which yields, in analogy with Equation \eqref{eq:inclusions},
$$2p^{d+n-1} + 1 \leq |B_{\mathbb{L}_{p,n}}(p^d)| \leq 2p^{d+n}+1,$$
Consequently, an argument similar to that given in Lemma \ref{lemma:boundeddoubling} yields 
$$|B_{\mathbb{L}_{p,n}}(p^{d+1})|< 2p^4(|B_{\mathbb{L}_{p,n}}(p^d)|).$$
In particular, we see that the inequalities described in Equation \eqref{eq:anotherkeydoublingineq} still hold with $\mathbb{L}_p$ replaced by $\mathbb{L}_{p,n}$.  Since the same reasoning holds for each $R \geq 1$, $\frac{1}{p^n}\mathbb{Z}$ has the property of bounded $p$-dilation with respect to $\mathbb{L}_{p,n}$.  As in the setting of $\mathbb{Z}\Big[\frac{1}{p}\Big]$, the conclusion that $\mathbb{L}_p$ is also of bounded doubling when restricted to $\frac{1}{p^n}\mathbb{Z}$ follows from the fact that $|B_{\mathbb{L}_{p,n}}(2R)| \leq |B_{\mathbb{L}_{p,n}}(pR)|$ for $R \geq 1$.
\end{proof}
$\newline$
Proposition \ref{prop:lengthfcn sum}, together with Lemma \ref{lemma:restriction bounded doubling}, now yields the following corollary.
\begin{cor}\label{cor:restriction length function sum}
Fix a prime $p$ and $n \in \mathbb{N}$.  Let $\mathbb{L}_{p,n}^{\Sigma}: \Gamma_n=\frac{1}{p^n}\mathbb{Z} \times \frac{1}{p^n}\mathbb{Z} \to [0, \infty)$ be given by
$$\mathbb{L}_{p,n}^{\Sigma}(\gamma_1, \gamma_2) = \mathbb{L}_{p,n}(\gamma_1) + \mathbb{L}_{p,n}(\gamma_2).$$
The function $\mathbb{L}_{p,n}^{\Sigma}$ is a proper length function on $\Gamma_n$ that is of bounded doubling with $C_{\mathbb{L}_{p,n}^{\Sigma}}=(4p^8)^4$.
\end{cor}
$\newline$
The proof of Corollary \ref{cor:restriction length function sum} can also be adapted to a more general setting.
\begin{cor}\label{cor:generalized restriction length function sum}
Let $\Gamma_1$ and $\Gamma_2$ each be countable discrete  groups and $\Gamma = \Gamma_1 \times \Gamma_2$.  Define $\mathbb{L}^{\Sigma}$ as in Equation \eqref{eq:length func}. For each $i \in \{1, 2\}$, suppose that $G_i$ is a subgroup of  $\Gamma_i$.  Denote the restriction of $\mathbb{L}^{(i)}$ to $G_i$ by $\mathbb{L}_{G_i}^{(i)}$.  If each $\mathbb{L}_{G_i}^{(i)}$ has the bounded doubling property with $C_{\mathbb{L}_{G_i}^{(i)}}$, then the restriction of $\mathbb{L}^{\Sigma}$ to $G_1 \times G_2$ is of bounded doubling with constant given by $\Big[\, max\Big\{C_{\mathbb{L}_{G_1}^{(1)}}, C_{\mathbb{L}_{G_2}^{(2)}} \Big\} \, \Big]^4 $.
\end{cor}
As a consequence of Corollary \ref{cor:restriction length function sum}, our Dirac operators for noncommutative solenoids can be used to define Dirac operators for rotation algebras.  Given $C^*\Big(\mathbb{Z}\Big[\frac{1}{p}\Big]\times \mathbb{Z}\Big[\frac{1}{p}\Big], \sigma_{\theta}\Big)$, recall from Subsection 2.2 that $(\sigma_{\theta})_n$ is the restriction of the multiplier $\sigma_{\theta}$ to the subgroup $\Gamma_n=\frac{1}{p^n}\mathbb{Z} \times \frac{1}{p^n}\mathbb{Z}$.  For the representation of $C^*\Big(\Gamma_n, (\sigma_{\theta})_n\Big)$ as bounded operators on the Hilbert space ${H}_n :=\ell^2\big(\Gamma_n\big)$, set $\lambda_{(\sigma_{\theta})_n}$ equal to the left $(\sigma_{\theta})_n$-regular  representation of $C^*\Big(\Gamma_n, (\sigma_{\theta})_n\Big)$ on $H_n$.  For every $n \in \mathbb{N}$, there exists a $^*$-isomorphism $\upsilon_{n}$ which takes $C^*(\mathbb{Z}^2, \sigma_{\theta_{2n}})$ to $C^*(\Gamma_n, (\sigma_{\theta})_n)$.  Then $\pi_{\theta_{2n}} : C^*(\mathbb{Z}^2, \sigma_{\theta_{2n}}) \to \mathcal{B}(H_n)$ defined by
$$\pi_{\theta_{2n}} = \lambda_{(\sigma_{\theta})_n}\circ \upsilon_{n}$$
gives a representation of $ C^*(\mathbb{Z}^2, \sigma_{\theta_{2n}})$ that is unital and faithful. Also using $\upsilon_n$ to denote the abstract isomorphism between the groups $\mathbb{Z}^2$ and  $\Gamma_n,$ we let ${D}_{p,n}$ be the operator $M_{\mathbb{L}_{p,n}^{\Sigma}\circ\upsilon_n }$. We note that for every $n\in\mathbb N,$ the length function $\mathbb{L}_{p,n}^{\Sigma}\circ\upsilon_n$ on $\mathbb{Z}^2$  is of bounded doubling, since $\mathbb{L}_{p,n}^{\Sigma}$ is of bounded doubling on $\Gamma_n$ by Corollary \ref{cor:restriction length function sum} .  Since unbounded length functions on groups that are of bounded doubling are also proper, the same kinds of arguments as used in the proof of Theorem \ref{thm:ncsolenoidspectraltriple} can be applied to demonstrate that ${D}_{p,n}$ is a Dirac operator.  
\begin{prop}\label{thm:quantumtorusqcms}
Fix a prime $p$ and $n \in \mathbb{N}$.   For every $\theta \in \Omega_p$ and $n \in \mathbb{N}$, $(C^*(\mathbb{Z}^2, \sigma_{\theta_{2n}}), \ell^2(\Gamma_n), {D}_{p,n})$ with representation $\pi_{\theta_{2n}}$ is a spectral triple for the noncommutative torus $C^*(\mathbb{Z}^2, \sigma_{\theta_{2n}}),$ with associated smooth subalgebra $C_C(\mathbb{Z}^2, \sigma_{\theta_{2n}}).$ 
\end{prop}
We next consider the topology induced by $\text{mk}_{L_{D_{p,n}}}$ on $\mathcal{S}(C^*(\mathbb{Z}^2, \sigma_{\theta_{2n}})),$ where $L_{D_{p,n}}$ is the seminorm on $C_C(\mathbb{Z}^2, \sigma_{\theta_{2n}})$ induced by the Dirac operator $D_{p,n}.$
When equipped with the weak$^*$-topology, the state space of a unital $C^*$-algebra is compact and second countable.  The same argument given in the proof of Theorem \ref{thm:ncsolenoidqcms} and Corollary \ref{cor:restriction length function sum} will still hold here, and  the discussion of Remark \ref{rmk:uniformlyeqmetrics} applies to show that length functions on $\mathbb{Z}^2$  that are of bounded doubling will give rise to uniformly equivalent metrics on  $\mathcal{S}(C^*(\mathbb{Z}^2, \sigma_{\theta_{2n}})),$ since they both determine the same compact weak-$\ast$ topology on $\mathcal{S}(C^*(\mathbb{Z}^2, \sigma_{\theta_{2n}}))$.  Therefore we obtain:
\begin{prop}\label{thm:quantumtorusqcms}
Fix a prime $p$ and $n \in \mathbb{N}$.  For every $\theta \in \Omega_p$, $(C^*(\mathbb{Z}^2, \sigma_{\theta_{2n}}), L_{D_{p,n}})$ is a Leibniz quantum compact  metric space.   
\end{prop}

\subsection{Spectral Triples on Noncommutative Solenoids Produced as Limits of Spectral Triples on Noncommutative Tori}

We now use the set-up of the previous section to build a countable inductive system of spectral triples for rotation algebras, in the sense of Floricel and Ghorbanpour (see Definition \ref{def:FGinductivelim}), which will satisfy the hypotheses of Theorems \ref{mainthm-FG}	and \ref{mainthm2-FG}.  Fix a prime $p$.  Given $\theta \in \Omega_p$, recall the construction of $\Big\{ (C^*(\mathbb{Z}^2, \sigma_{\theta_{2n}}) \Big\}_{n \in \mathbb{N}}$, where the multiplier  $\sigma_{\theta}$ on $\mathbb{Z}\Big[\frac{1}{p}\Big] \times \mathbb{Z}\Big[\frac{1}{p}\Big]$ is chosen to construct the noncommutative solenoid $\mathcal{A}_{\theta}^{\mathcal S}$. 

Again, throughout this subsection, denote $\mathbb{Z}\Big[\frac{1}{p}\Big]\times\mathbb{Z}\Big[\frac{1}{p}\Big]$ by $\Gamma.$ For each $n\in \mathbb N,$ denote the subgroup $\frac{1}{p^n} \mathbb{Z}\times \frac{1}{p^n} \mathbb{Z}$ of $\Gamma$ by $\Gamma_n.$  Since  $\Gamma_n$ embeds into $\Gamma_{n+1}$ via inclusion maps and $(\sigma_{\theta})_n$ is the restriction of $\sigma_{\theta}$ from $\Gamma$ to  $\Gamma_n$, there exist, for every $j,k \in \mathbb{N}$, $j \leq k$,  unital $\ast$-monomorphisms  
$$\beta_{j,k}: C^*(\Gamma_j, (\sigma_{\theta})_j) \to  C^*(\Gamma_k,(\sigma_{\theta})_k).$$
Let $\phi_{j,k}: C^*(\mathbb{Z}^2, \sigma_{\theta_{2j}}) \to  C^*(\mathbb{Z}^2, \sigma_{\theta_{2k}})$ be the map given by the $^*$-monomorphism
$$ \phi_{j,k} = \upsilon_{k}^{-1} \circ \beta_{j,k} \circ \upsilon_{j}$$
and 
$$I_{j,k}: H_j \to  H_k$$
be the isometry given by the inclusion map, where, as in previous sections, $\upsilon_j$ is the $\ast$-isomorphism between $C^*(\mathbb{Z}^2, \sigma_{\theta_{2j}})$ and $C^*(\Gamma_j, (\sigma_{\theta})_j ),$ and we use the notation $H_j=\ell^2(\Gamma_j),\;j\in\mathbb N.$  Since $C_C(\mathbb{Z}^2, \sigma_{{2j}})$ is equal to $\upsilon_{j}^{-1}\Big(C_C(\Gamma_j, (\sigma_{\theta})_j)\Big)$ and $\beta_{j,k}\Big(C_C(\Gamma_j, (\sigma_{\theta})_j)\Big) \subseteq C_C(\Gamma_k, (\sigma_{\theta})_k)$,
\begin{equation}
	\label{eq:inclusiondirlim}
\phi_{j,k}(C_C(\mathbb{Z}^2, \sigma_{\theta_{2j}}) \subseteq C_C(\mathbb{Z}^2, \sigma_{\theta_{2k}}).
\end{equation}
Moreover, for every $f \in C^*(\mathbb{Z}^2, \sigma_{\theta_{2j}})$:
$$I_{j,k}\circ \pi_{\theta_{2j}}(f) = \pi_{\theta_{2k}}(\phi_{j,k}(f))\circ  I_{j,k}.$$
 By definition, $\mathbb{L}_{p,n}^{\Sigma}$ agrees with $\mathbb{L}_p^{\Sigma}$ restricted to $\Gamma_n$ for each $n \in \mathbb{N}$, so that our choice of $I_{j,k}$ yields  $I_{j,k}(\text{dom}(M_{\mathbb{L}_{p,j}^{\Sigma}})) \subseteq \text{dom}(M_{\mathbb{L}_{p,k}^{\Sigma}})$.  In particular, 
$$I_{j,k}D_{p,j} = D_{p,k} I_{j,k}.$$  Note that by construction, $\phi_{k,l} \phi_{j,k} = \phi_{j,l}$ and $I_{k,l} I_{j,k} = I_{j,l}$ for all $j,k,l \in \mathbb{N}, j \leq k \leq l$.  Therefore, 
\begin{equation*}
    \{ \, (C^*(\mathbb{Z}^2, \sigma_{\theta_{2j}}), H_j),\; (\phi_{j,k}, I_{j,k}) \, \}_{j \in \mathbb{N}}\;\text{with associated smooth subalgebras}\;\{ C_C(\mathbb{Z}^2, \sigma_{\theta_{2j}}) \}_{j \in \mathbb{N}},
\end{equation*} 
is a countable inductive system of spectral triples (see Definitions \ref{def:FGmorphism}, \ref{def:FGinductivelim}, \ref{def:FGspectraltriple}).  

We now are in a position to deduce that the spectral triples for noncommutative solenoids constructed in Theorem \ref{thm:ncsolenoidspectraltriple} can each be obtained as the inductive realization of a countable inductive system of finite-dimensional spectral triples.  For the inductive system of Hilbert spaces, take $\Big\{\Big(H_j=\ell^2(\Gamma_j), I_{j,k} \Big)  \Big\}_{j \in \mathbb{N}}$.  Then
$$\ell^2(\Gamma) = \lim_{\substack{\longrightarrow \\ j \in J}} H_j.$$
By the way that the morphisms $\phi_{j,k}$ are defined, the inductive limit of the inductive system of $C^*$-algebras $\{ (C^*(\mathbb{Z}^2, \sigma_{\theta_{2j}}), \phi_{j,k}) \}_{j \in \mathbb{N}}$ is $^*$-isomorphic to that of $\Big\{ \Big(C^*(\Gamma_j, (\sigma_{\theta})_j), \beta_{j,k}\Big) \Big\}_{j \in \mathbb{N}}$.  The inductive limit of the second inductive system is $C^*(\Gamma, \sigma_{\theta})=\mathcal{A}_{\theta}^{\mathcal S},$ with associated smooth subalgebra $C_C(\Gamma, \sigma_{\theta}).$  Consequently, $\varinjlim_{j\in\N} C^*(\mathbb{Z}^2, \sigma_{\theta_{2j}})$ is also $^*$-isomorphic to $\mathcal{A}_{\theta}^{\mathcal S}$ with associated smooth subalgebra $C_C(\Gamma, \sigma_{\theta}).$ We denote this $^*$-isomorphism by $\rho:\varinjlim_{j\in\N} C^*(\mathbb{Z}^2, \sigma_{\theta_{2j}}) \to \mathcal{A}_{\theta}^{\mathcal S}.$ We calculate that at the level of representations on $\varinjlim_{j\in\N} C^*(\mathbb{Z}^2, \sigma_{\theta_{2j}}),$
$$\pi_{\theta} = \lambda_{\sigma_{\theta}}\circ \rho = \lim_{\substack{\longrightarrow \\ j \in J}} \lambda_{\sigma_{\theta_{2j}}} \circ \upsilon_{j}=  \lim_{\substack{\longrightarrow \\ j \in J}} \pi_{\theta_{2j}}.  $$
Thus we obtain a representation of $\varinjlim_{j\in\N} C^*(\mathbb{Z}^2, \sigma_{\theta_{2j}})\cong\mathcal{A}_{\theta}^{\mathcal S}$ that is an inductive limit of the family of representations $\{ \pi_{\theta_{2j}} \}_{j \in \mathbb{N}}$ associated with the family of spectral triples $\{ (C^*(\mathbb{Z}^2, \sigma_{\theta_{2j}}), H_j, D_{p,j}) \}_{j \in \mathbb{N}}$. By its very definition, the $\ast$-isomorphism $\rho:\varinjlim_{j\in\N} C^*(\mathbb{Z}^2, \sigma_{\theta_{2j}}) \to \mathcal{A}_{\theta}^{\mathcal S}$ intertwines the $C^*$-algebra representation pairs $(\varinjlim_{j\in\N} C^*(\mathbb{Z}^2, \sigma_{\theta_{2j}}),\;\pi_{\theta})$ and $(C^*(\Gamma, \sigma_{\theta}), \lambda_{\sigma_{\theta}}).$  For each fixed $j \in \mathbb{N}$, set $I_j :H_j \to H=\ell^2 (\Gamma)$ equal to $\displaystyle{\lim_{\substack{\longrightarrow \\ k \in J}}} \, I_{j,k}$ and $\phi_j: C^*(\mathbb{Z}^2, \sigma_{\theta_{2j}}) \to \mathcal{A}_{\theta}^{\mathcal S}$ equal to $\displaystyle{\lim_{\substack{\longrightarrow \\ k \in J}}} \, \beta_{j,k}\circ \upsilon_{j}.$  The combination of these choices yield $\pi_{\theta}$ as the unique representation of $\varinjlim_{j\in\N} C^*(\mathbb{Z}^2, \sigma_{\theta_{2j}})$ on $\ell^2(\Gamma)$ such that for every $f \in C^*(\mathbb{Z}^2, \theta_{2j})$,
$$ \pi_{\theta}(\phi_j(f))\circ I_j = I_j\circ  \pi_{\theta_{2j}}(f).$$
To see that $\mathcal{D}_p = \displaystyle{\lim_{\substack{\longrightarrow \\ j \in J}}} \, D_{p,j}$, observe that $\mathcal{D}_p$ agrees with $D_{p,j}$ on  $\ell^2 (\Gamma_j)$.  Hence $(\mathcal{A}_{\theta}^{\mathcal S}, \ell^2(\Gamma), \mathcal{D}_p)$ is the inductive realization of our inductive system of spectral triples for rotation algebras.  

Since noncommutative solenoids are inductive limit $C^*$-algebras, our spectral triples for noncommutative solenoids can be viewed as ``inductive limit spectral triples'' on inductive limit $C^*$-algebras.    Because all of the test cases examined by Floricel and Ghorbanpour involve countable inductive systems of finite-dimensional spectral triples whose inductive realizations are spectral triples for $AF$-algebras, our spectral triples supply important new examples where Floricel and Ghorbanpour's results can be applied.  In particular, our spectral triples for noncommutative solenoids, when viewed as inductive realizations of countable inductive systems of spectral triples for rotation algebras,  satisfy the requirements for spectral triple structure given in Theorems 3.1 and 3.2 of \cite{FlGh}, as we show in the following theorem.

\begin{thm}\label{thm:FGinductivelim}
Fix a prime $p$.  Let $\Gamma = \mathbb{Z} \Big[\frac{1}{p}\Big] \times \mathbb{Z} \Big[\frac{1}{p}\Big]$ and for each $n \in \mathbb{N}$, set $\Gamma_n = \frac{1}{p^n} \mathbb{Z} \times \frac{1}{p^n} \mathbb{Z}$.  For every $\theta \in \Omega_p$, the triple $(\mathcal{A}_{\theta}^{\mathcal S}, \ell^2(\Gamma), \mathcal{D}_p)$ with associated smooth subalgebra $C_C(\Gamma,\sigma_{\theta})$ can be written as the inductive realization of 
\begin{equation*}
    \{ \, (C^*(\mathbb{Z}^2, \sigma_{\theta_{2j}}), \ell^2(\Gamma_j), D_{p,j}), (\phi_{j,k}, I_{j,k}) \, \}_{j \in \mathbb{N}},\;\text{with associated smooth subalgebras}\;\{(C_C(\mathbb{Z}^2, \sigma_{\theta_{2j}})\}_{j \in \mathbb{N}}.
\end{equation*} 
Furthermore, the above inductive system of spectral triples satisfies the hypotheses of Floricel and Ghorbanpour's Theorems \ref{mainthm-FG} and \ref{mainthm2-FG}, so that the inductive realization  $(\mathcal{A}_{\theta}^{\mathcal S}, \ell^2(\Gamma), \mathcal{D}_p)$ of the inductive system is itself a spectral triple, with associated smooth subalgebra $C_C(\Gamma, \sigma_{\theta}).$ 
\end{thm}
\begin{proof}
To show that $\mathcal{D}_p$, as the inductive limit of the family of operators $\{ D_{p,j} \}_{j \in \mathbb{N}}$, has compact resolvent, fix $t \in \mathbb{R} \setminus \{ 0 \}$.  By Theorem 3.1 of $\cite{FlGh}$, it suffices to show that the sequence of operators $\{ I_j(D_{p,j} - it)^{-1}I_j^* \}_{j \in \mathbb{N}}$ converges in norm to $(\mathcal{D}_p-it)^{-1}$.  Choose an $\epsilon >0$.  Then there exists $R>0$ such that for $z \in \mathbb{R}$, $z > R$ implies $|z - it|^{-1} < \epsilon$.  Given such an $R$, define $B_{\mathbb{L}_p^{\Sigma}}(R)$ as in part (1) of Definition \ref{def:lengthfcnbddoubl}, where we recall that $\mathbb{L}_p^{\Sigma}$ was defined in Corollary \ref{cor:restriction length function sum}.  As in the proof of Theorem \ref{thm:ncsolenoidspectraltriple}, $\mathbb{L}_p^{\Sigma}$ is proper implies $B_{\mathbb{L}_p^{\Sigma}}(R)$ is a finite set.  Let $\{ \gamma_k \}_{1 \leq k \leq m}$ be an enumeration of the elements in this set.  For each $1 \leq k \leq m$, find $n_k \in \mathbb{N}$ and $a_k \in \mathbb{Z}$ such that $\gamma_k = \frac{a_k}{p^{n_k}}$ with $\gcd(a_k, p) = 1$.  Set $N = \max \{ n_1, \cdots, n_m \}$ and define $B_{\mathbb{L}_{p}^{\Sigma}}(R)$ be the ball of radius $R$ in the length metric in $\Gamma$ as in part (1) of Definition \ref{def:lengthfcnbddoubl}.  Note that $B_{\mathbb{L}_p^{\Sigma}}(R) \subset \Gamma_N$.  For each $1 \leq k \leq m$, now let $P_k$ denote the orthogonal projection taking $\ell^2(\Gamma_N)$ onto the one-dimensional subspace spanned by $\delta_{\gamma_k}$.  Then $I_N P_k I_N^*$ describes the orthogonal projection taking  $\ell^2(\Gamma)$ onto the one-dimensional subspace spanned by $\delta_{\gamma_k}$ viewed as an element of $\ell^2(\Gamma)$.  Since each $\delta_{\gamma_k}$ is an eigenvector of $(\mathcal{D}_p - it)^{-1}$ with eigenvalue $(\mathbb{L}_p^{\Sigma}(\gamma_k) - it)^{-1}$, 
$j > N$ implies, by the triangle inequality,

\[
\begin{split}
\| \, (\mathcal{D}_p-it)^{-1} - &I_j(D_{p,j}-it)^{-1}I_j^* \, \|_{{\mathcal{B}}(\ell^2(\Gamma))}\; \\
 &\leq \Big\| \, (\mathcal{D}_p - it)^{-1} - I_N \Big(\sum_{\gamma_k \in B_{\mathbb{L}_p^{\Sigma}}(R)}(\mathbb{L}_p^{\Sigma}(\gamma_k)-it)^{-1}P_k\Big)I_N^* \, \Big\|_{{\mathcal{B}}(\ell^2(\Gamma))} \qquad \qquad \qquad \qquad \qquad \qquad\\
&+  \Big\| \, I_N \Big(\sum_{\gamma_k \in B_{\mathbb{L}_p^{\Sigma}}(R)}(\mathbb{L}_p^{\Sigma}(\gamma_k)-it)^{-1}P_k\Big)I_N^* - I_j(D_{p,j}-it)^{-1}I_j^* \, \Big\|_{{\mathcal{B}}(\ell^2(\Gamma))}\\
&= \Big\| \, (\mathcal{D}_p - it)^{-1} \quad - \sum_{\gamma_k \in B_{\mathbb{L}_p^{\Sigma}}(R)}(\mathbb{L}_p^{\Sigma}(\gamma_k)-it)^{-1}I_N P_k I_N^* \, \Big\|_{{\mathcal{B}}(\ell^2(\Gamma))}\qquad \qquad \qquad \qquad \qquad \qquad\\  
&+  \Big\| \, I_j \Big(I_{N,j} \Big(\sum_{\gamma_k \in B_{\mathbb{L}_p^{\Sigma}}(R)}(\mathbb{L}_p^{\Sigma}(\gamma_k)-it)^{-1}P_k\Big)I_{N, j}^* - (D_{p,j}-it)^{-1}\Big)I_j^* \, \Big\|_{{\mathcal{B}}(\ell^2(\Gamma))}\\
&= \Big\| \, (\mathcal{D}_p - it)^{-1} \quad - \sum_{\gamma_k \in B_{\mathbb{L}_p^{\Sigma}}(R)}(\mathbb{L}_p^{\Sigma}(\gamma_k)-it)^{-1}I_N P_k I_N^* \, \Big\|_{{\mathcal{B}}(\ell^2(\Gamma))}\qquad \qquad \qquad \qquad \qquad \qquad\\ 
&+\Big\| \, \sum_{\gamma_k \in B_{\mathbb{L}_p^{\Sigma}}(R)}(\mathbb{L}_p^{\Sigma}(\gamma_k)-it)^{-1} I_{N,j} P_k I_{N, j}^* - (D_{p,j}-it)^{-1}\Big) \, \Big\|_{{\mathcal{B}}(\ell^2(\Gamma))}\\
&=\sup \left\{ |\mathbb{L}_p^{\Sigma}(\gamma) - it|^{-1} : \gamma \notin B_{\mathbb{L}_p^{\Sigma}}(R)  \right\}  +  \sup \left\{ |\mathbb{L}_p^{\Sigma}(\gamma) - it|^{-1} : \gamma \notin B_{\mathbb{L}_p^{\Sigma}}(R)  \right\} < 2 \epsilon, 
\end{split}
\]
where the bound on the second term of the last equality holds because each $\delta_{\gamma_k}$ is also an eigenvector of $D_{p,j}$ with eigenvalue $(\mathbb{L}_p^{\Sigma}(\gamma) - it)^{-1}$ if $j > N$.  As the choice of $\epsilon$ was arbitrary, the sequence $\{ I_j(D_{p,j} - it)^{-1}I_j^* \}_{j \in \mathbb{N}}$ converges uniformly to $(\mathcal{D}_p-it)^{-1}$.  Hence the inductive realization of $\{ \, (C^*(\mathbb{Z}^2, \sigma_{\theta_{2j}}), \ell^2(\Gamma_j), D_{p,j}), (\phi_{j,k}, I_{j,k}) \, \}_{j \in \mathbb{N}}$ satisfies condition (ST1) in Definition \ref{def:spectraltriple} for a Dirac operator.  

To see that $\mathcal{D}_p$, as the inductive limit of the family of operators $\{ D_{p,j} \}_{j \in \mathbb{N}}$, satisfies the remaining denseness condition given in (ST2) of Definition \ref{def:spectraltriple}, fix $J \in \mathbb{N}$ and a choice of $g \in C_C(\mathbb{Z}^2, \theta_{2J})$.  By Theorem \ref{mainthm2-FG}, it suffices to show that the family of commutators $\{ \, [D_{p,k}, \, \pi_{\theta_{2k}}(\phi_{J,k}(g))] \, \}_{k \geq J}$ is uniformly bounded. The key norm inequality for commutators in Equation \eqref {inequal:normcommutator} shows that for all $g\in C_C(\mathbb{Z}^2, \theta_{2J}),$ 
\[
\| \, [D_{p,k}, \, \pi_{\theta_{2k}}(\phi_{J,k}(g))] \, \|_{{\mathcal B}(H_k)}\;=\;\| \, [D_{p,k}, \, \lambda_{\sigma_{\theta_{2k}}}\circ \upsilon_k(\phi_{J,k}(g))] \, \|_{{\mathcal B}(H_k)}\leq\;\sum_{\gamma\in \Gamma_k }\mathbb{L}_{p,k}^{\Sigma}(\gamma)|\upsilon_k(\phi_{J,k}(g))(\gamma)|.
\]
Since $\mathbb{L}_{p,k}^{\Sigma}$ agrees with $\mathbb{L}_{p,J}^{\Sigma}$ on $\Gamma_J \subseteq \Gamma_k$ for each $k \geq J$, and since both agree with the restriction of $\mathbb{L}_{p,}^{\Sigma}$ to the respective subgroups, we obtain 
$$\| \, [D_{p,k}, \, \pi_{\theta_{2k}}(\phi_{J,k}(g))] \, \|_{{\mathcal B}(H_k)}\; \leq\; \sum_{\gamma\in\Gamma}\mathbb{L}_{p}^{\Sigma}(\gamma)|\phi_{J}(g)(\gamma)|,\;\forall\;k\geq J$$
which, in combination with the observation that $g\in C_C(\mathbb{Z}^2, \theta_{2J})$ implies $\phi_J(g)\in\;C_C(\Gamma_J, (\sigma_{\theta})_J)\;\subset\; C_C(\Gamma, \sigma_{\theta}),$ proves the desired uniform boundedness condition.
\end{proof}

\begin{rmk} 
If we had tried to take direct limits of Dirac operators arising from the standard length function on $\mathbb Z^2$ and $\Gamma_j=\frac{1}{p^j}\mathbb Z^2$ as described in Example \ref{example_spectraltriplegroupalg}, we would not have obtained this result, because some of the standard lengths of non-zero elements in $\Gamma_j$ go to $0$ as $j\to +\infty.$  This would have caused major problems if we tried to verify condition (ST1).  Fortunately, our adjusted lengths on $\Gamma_j$ are bounded from below by $1$ on non-zero elements.
\end{rmk}

\section{A Dense Subalgebra of Twisted Abelian Group $C^*$-Algebras with Spectral Invariance}

We will now examine dense subdomains associated to  our  spectral triples which arise from our Dirac operators in greater detail.  Notable properties to be discussed include identification, for fixed prime $p$ and $\theta \in \Omega_p$, of a dense $^*$-subalgebra of $\mathcal{A}_{\theta}^{\mathcal S}$ having the property of spectral invariance in the sense of L. Schweitzer \cite{Sch}, on which $L_{\mathcal{D}_p}$ is finite. 
For similar $\ast$-subalgebras in the $\sigma-$twisted rapid decay
case, see \cite{Mathai}. The existence of such a $^*$-subalgebra was conjectured in a question posed by Long and Wu at the conclusion of $\cite{Long-Wu2}$.  The existence of this additional condition on $L_{\mathcal{D}_p}$ was also described as ``useful'' by Christ and Rieffel in the discussion after Proposition 1.6 in \cite{Christ-Rieffel}.  To obtain such results, our efforts will rely heavily on methods from both noncommutative geometry and  noncommutative harmonic analysis. 

Our length function on the discrete abelian group $\Gamma = \mathbb{Z} \Big[\frac{1}{p}\Big] \times \mathbb{Z} \Big[\frac{1}{p}\Big]$ will be used to define a Fr\'echet $\ast$-algebra that is a dense $^*$-subalgebra having the property of spectral invariance in the twisted group $C^*$-algebra $C^*(\Gamma, \sigma_{\theta})$. Indeed, some of our results will apply to twisted group $C^*$-algebras associated to arbitrary countable discrete nilpotent groups.   As discussed by Christ and Rieffel in \cite{Christ-Rieffel}, subalgebras equipped with these properties also capture the $K$-theory of the $C^*$-algebra.  With this aim in mind, we recall:
\begin{defn}[Jolissaint, c.f. Definition 2.1 of \cite{jol}]\label{def:lengthfcnspace}
	Let $\Gamma$ be a countable discrete  group and $\sigma$ a multiplier on $\Gamma$.  Given a length function $\mathbb{L}$ on $\Gamma$, set \textbf{$H^{1,\infty}_{\mathbb L}(\Gamma,\sigma)$} equal to the {\bf space of complex-valued functions} $f$ on $\Gamma$  such that for each $s\geq 0,$ the function $f(1+{\mathbb L})^s$ belongs to $\ell^1(\Gamma, \sigma).$  For $f \in H^{1,\infty}_{\mathbb L}(\Gamma, \sigma )$ and $s\geq 0,$ define the norm $\| \cdot \|_{1,s,\mathbb{L}}: H^{1,\infty}_{\mathbb{L}}(\Gamma, \sigma) \to [0,\infty)$ by 
	$$\| f \|_{1,s,{\mathbb L}}\;=\;\sum_{\gamma \in \Gamma}|f(\gamma)|(1+{\mathbb L}(\gamma))^s.$$	
\end{defn}

\begin{rmk}\label{rm:mu}
	Notice that if we  set,  for fixed $q \in \N$ and for $\phi \in \ell^1(\Gamma, \sigma),$
	\begin{equation*}
	\mu_{q, \sigma}(\phi) = \sup_{N\in \N  }\left\{ N^q \, \| (1 - \chi_N) \phi \|_{\ell^1(\Gamma, \sigma)}\right\},
	\end{equation*}
	then the same proof  as  in  Lemma 2.2 of \cite{jol}  shows    that  $\phi \in H^{1, \infty}_{\mathbb{L}}(\Gamma, \sigma )$ if and only if $\mu_{q, \sigma}(\phi)\; <\; \infty,\;\forall q\in \mathbb N.$ 
	\end{rmk}
\begin{rmk}
Our definition is a modification of that given in \cite{jol}, which is specified for discrete group $C^*$-algebras rather than twisted discrete group $C^*$-algebras. Although the multiplier $\sigma$ does not play a role in the definition of the norm $\| \cdot \|_{1, s, \mathbb{L}}$, each such norm for $s \geq 0$ will be shown in the proof of Proposition \ref{prop:frechetalgebra} to exhibit  submultiplicativity with respect to the convolution product twisted by $\sigma$.  
\end{rmk}
\begin{rmk}\label{rmk:twisted Frechet algebra}
A Fr\'echet algebra which arises from a twisted group $C^*$-algebra and whose topology can be induced by a collection of seminorms which are each submultiplicative with respect to the twisted convolution product will be called a {\bf twisted Fr\'echet algebra}.  The collection of seminorms used in our definition of $H^{1,\infty}_{\mathbb{L}}(\Gamma, \sigma)$ happens to also be a collection of norms. 
\end{rmk}

We will demonstrate that $H^{1, \infty}_{\mathbb{L}}(\Gamma, \sigma_{}) \subset C^*(\Gamma, \sigma_{})$ is stable under the holomorphic functional calculus (see Theorem  \ref{prop:disc group multipl}).  We first introduce the notion of a spectral invariant subalgebra, as introduced by Schweitzer in \cite{Sch}.

\begin{defn} \cite{Sch}\label{def:SI}
Let $A$ be a unital *-subalgebra of a unital Banach algebra $B.$	We say that $A$ is a {\bf spectral invariant (SI) $\ast$-subalgebra} of $B$ if the invertible elements of $A$ are precisely those elements in $A$ that are invertible in $B.$ Note $A$ is SI in $B$ if and only if for every $a\in A,$ the spectrum of $a$ in $A$ is the same as the spectrum of $a$ in $B.$

\end{defn}

\begin{rmk}
	\label{rmk:Schweizer}
  As a consequence of Lemma 1.2 of Schweitzer in \cite{Sch}, a unital Fr\'echet $\ast$-subalgebra $A$ of a unital $C^*$-algebra {\bf $B$ is SI if and only if it is stable under the holomorphic functional calculus}, that is, if a function $f$ is holomorphic in a neighborhood of the spectrum of $a\in B,$ the element $f(a)$ of $B$ lies in $A$. This is important particularly when considering $K$-theoretic invariants of $B,$ because under these hypotheses, the $K$-theory of $A$ will be the same as the $K$-theory of $B.$
  \end{rmk}
  
 Working in the more general setting of countable  nilpotent groups, we will first apply ideas developed by Jolissaint in \cite{jol}, as well as techniques used by R. Ji in \cite{ji}, to show that $H^{1, \infty}_{\mathbb{L}}(\Gamma, \sigma)$ is a twisted Fr\'echet $\ast$-subalgebra that is dense and spectral invariant in $\ell^1(\Gamma, \sigma)$.  To deduce that $\ell^1(\Gamma, \sigma)$ is SI in $C^*(\Gamma, \sigma)$, we will rely on results obtained by Austad in \cite{Aust} building on the work of Gr\"ochenig and Leinert in \cite{Groch-Lein}, Ludwig in \cite{Ludwig}, and  Hulanicki in \cite{Hul}.  We will then be able to equip our spectral triples for noncommutative solenoids with $H^{1, \infty}_{\length}(\Gamma, \sigma_{\theta})$, $\Gamma = \mathbb{Z} \Big[\frac{1}{p}\Big] \times \mathbb{Z} \Big[\frac{1}{p}\Big]$, as a choice of dense smooth subalgebra (see Corollary \ref{cor: main thm}).

\subsection{Spectral Invariance of Dense Twisted Fr\'echet $\ast$-Subalgebras of $\ell^1(\Gamma, \sigma)$}

For twisted group $C^*$-algebras, the multiplier plays a role in adjoint and convolution product computations.  Since multipliers on groups take complex values of unit modulus, the twisted Fr\'echet $\ast$-subalgebra  $H^{1,\infty}_{\mathbb L}(\Gamma,\sigma) \subset \ell^1(\Gamma, \sigma)$ can be shown to exhibit additional structure with respect to these operations. For generalities on Fr\'echet algebras in our context, see e.g. \cite{ji}.

\begin{prop}\label{prop:frechetalgebra}  Let $\Gamma$ be a countable discrete group and $\sigma$ a multiplier on $\Gamma$.  Suppose that $\mathbb{L}$ is a length function on $\Gamma$.  Then $H^{1,\infty}_{\mathbb L}(\Gamma,\sigma)$ is a (twisted) $\ast$-subalgebra of $\ell^1(\Gamma,\sigma)$.  Furthermore, $H^{1,\infty}_{\mathbb L}(\Gamma,\sigma)$ is also a {\bf twisted Fr\'echet $\ast$-algebra} with respect to the topology defined by the increasing sequence of norms $\{\|\cdot\|_{1,s,{\mathbb L}}\}_{s\geq 0}$. 
\end{prop}	

\begin{proof}
To show that $H^{1,\infty}_{\mathbb L}(\Gamma,\sigma)$ is invariant under the $^*$-operation in $\ell^1(\Gamma,\sigma)$, suppose that $f$ is in $H^{1,\infty}_{\mathbb L}(\Gamma,\sigma)$.  Then

\[
\begin{split}
\|f^*\|_{1,s,{\mathbb L}}\;&=\;
\sum_{\gamma \in \Gamma} | \, \overline{f(\gamma^{-1}) \sigma(\gamma,\gamma^{-1})} \, |(1+{\mathbb L}(\gamma))^s = 
\sum_{\gamma \in \Gamma} | f(\gamma)|(1+{\mathbb L}(\gamma))^s = \|f\|_{1,s,{\mathbb L}} \, .
\end{split}
\]
 Since the choice of $f$ was arbitrary, $H^{1,\infty}_{\mathbb L}(\Gamma,\sigma) \subset \ell^1(\Gamma,\sigma)$ is a $^*$-subalgebra.

To determine bounds on the behavior of $H^{1,\infty}_{\mathbb L}(\Gamma,\sigma)$ under the $\sigma$-twisted convolution product in $\ell^1(\Gamma,\sigma)$, a bound on the length function will first be obtained.  For every $\gamma, \eta \in \Gamma$, the fact that  $\mathbb{L}$ is a length function on $\Gamma$ yields, by taking $\gamma= \eta \eta^{-1}\gamma:$

\begin{equation}
\label{eq:threeStar}
\begin{split}
 1 + \mathbb{L}(\gamma) &\leq 1 + \mathbb{L}(\eta) + \mathbb{L}(\eta^{-1}\gamma) \leq  1 + \mathbb{L}(\eta) + \mathbb{L}(\eta^{-1}\gamma)+ L(\eta) L(\eta^{-1}\gamma)\\
&=\big(  1 + \mathbb{L}(\eta)\big)\big(  1 + \mathbb{L}(\eta^{-1}\gamma)\big).
\end{split}
\end{equation}
Thus for every $f, f' \in H^{1,\infty}_{\mathbb L}(\Gamma,\sigma)$, we have by Equation \eqref{eq:threeStar}:
\[
\begin{split}
\|( f \ast_{\sigma} f') \|_{1,s,{\mathbb L}} &= \sum_{\gamma\in \Gamma} | \, (f *_{\sigma} f') (\gamma) \, |(1+{\mathbb L}(\gamma))^s \leq \sum_{\gamma \in \Gamma}  \sum_{\eta \in \Gamma} | \, f (\eta) f' (\eta^{-1} \gamma) \sigma(\eta,\eta^{-1}\gamma )\, |(1+{\mathbb L}(\gamma))^s \\
&\leq \sum_{\gamma \in \Gamma}  \sum_{\eta \in \Gamma} | \, f (\eta) f' (\eta^{-1} \gamma) \,\big(  1 + \mathbb{L}(\eta)\big)^s \big(  1 + \mathbb{L}(\eta^{-1}\gamma)\big)^s  
\\
&= \sum_{\eta \in \Gamma} | f (\eta)| (1 + \mathbb{L}(\eta))^s \sum_{\gamma \in \Gamma} \, | f' (\eta^{-1}\gamma )| (1 + \mathbb{L}(\eta^{-1} \gamma))^s   
 =   \|f \|_{1,s,{\mathbb L}} \, \|f ' \|_{1,s,{\mathbb L}} < \infty.
\end{split}
\]
 Thus for each $s \geq 0$, $\| \cdot \|_{1,s, \mathbb{L}}$ has been shown to be submultiplicative with respect to the twisted convolution product.  Hence $H^{1,\infty}_{\mathbb L}(\Gamma,\sigma)$ is also a twisted Fr\'echet $\ast$-subalgebra of $\ell^1(\Gamma, \sigma)$.
\end{proof}
Because $C_C(\Gamma, \sigma) \subset H^{1, \infty}_{\mathbb{L}}(\Gamma)$, $H^{1, \infty}_{\mathbb{L}}(\Gamma)$ is dense in  $\ell^1(\Gamma, \sigma)$.  We now begin to lay the groundwork towards extending a result of Jolissaint in \cite{jol} to these twisted Fr\'echet $\ast$-subalgebras. 
\begin{prop} 
	\label{prop:joliThm}
	Let $\Gamma$ be a countable discrete group and $\sigma$ a multiplier on $\Gamma$.  Suppose that $\mathbb{L}$ is a length function on $\Gamma$. Then the twisted Fr\'echet $\ast$-subalgebra $H^{1,\infty}_{\mathbb L}(\Gamma,\sigma)$ is SI and dense in $\ell^1(\Gamma,\sigma)$.	

\end{prop}

\begin{proof}

In this proof, we follow the proof of  Jolissaint's Proposition 2.3  \cite{jol}. To check  Definition \ref{def:SI}, we  first consider the case in which   $H^{1,\infty}_{\mathbb L}(\Gamma,\sigma) \ni h=(1-f)  $, with  $\| f \|_{\ell^1(\Gamma, \sigma)} < \frac{1}{2}.$ Hence $h$ is first of all invertible in $\ell^1(\Gamma, \sigma)$, with inverse $h^{-1}=\sum_{n=0}^{\infty} f^n$, where exponentiation signifies repeated application of the twisted convolution product. We will show below  that  $h^{-1}  \in H^{1, \infty}_{\mathbb{L}}(\Gamma, \sigma)$, as needed for spectral invariance. Hence our  goal is to  obtain a bound on $\mu_{q, \sigma}(\sum_{n=1}^{\infty} f^n)$, which will imply $\sum_{n=0}^{\infty} f^n \in H^{1, \infty}_{\mathbb{L}}(\Gamma, \sigma )$, see Remark \ref{rm:mu}.

 Given $N \in \mathbb{N}$, define $B_{\mathbb{L}}(N)$ as in part (1) of Definition \ref{def:lengthfcnbddoubl}. Let $\chi_{N}$ denote the characteristic function supported on ${B_{\mathbb{L}}(N)}$.  Fix $q \in \mathbb{N}$. (Note that we  changed Jolissaint's notation from $\chi^N$ to $\chi_N$ for the characteristic function of $B_{\mathbb{L}}(N)$. Now fix $N\in\mathbb N,\;N \geq 1$.  For $\N \ni n \leq \sqrt{N}$,

\begin{equation}
\label{eq:joli1}
\begin{split}
(1 - \chi_{N})f^n &= (1 - \chi_N)[((1 - \chi_{\sqrt{N}})f) *_{\sigma} f^{n-1}] + (1 - \chi_N)[(\chi_{\sqrt{N}}f) *_{\sigma}((1 - \chi_{\sqrt{N}})f) *_{\sigma} f^{n-2}] + \cdots \\
 &\cdots + (1 - \chi_N)[(\chi_{\sqrt{N}}f)^k *_{\sigma}((1 - \chi_{\sqrt{N}})f) *_{\sigma} f^{n-k-1}] + \cdots \qquad \qquad\\
    &\cdots + (1 - \chi_N)[(\chi_{\sqrt{N}}f)^{n-1} *_{\sigma}((1 - \chi_{\sqrt{N}})f)]  +  (1 - \chi_N)(\chi_{\sqrt{N}}f)^n .
 \end{split} 
 \end{equation}  
This equality can be confirmed using the associative structure of $H^{1, \infty}_{\mathbb{L}}(\Gamma, \sigma)$ in combination with a telescoping/cancellation argument involving the expansion of  $f^n$.

Now note that  the last  term is zero because $(1-\chi_{N})$ is supported outside the ball of radius $N$.  Consequently we only have the first  $n$ terms left.  Using these $n$ terms, the triangle inequality, and the submultiplicativity of the norm, we obtain, for $n\geq 1:$

\begin{equation}
\label{eq:joli-ineq}
\|(1 - \chi_{N})f^n\|_{\ell^1(\Gamma,\sigma)}\;\leq\; \|n((1 - \chi_{\sqrt{N}})f)\|_{\ell^1(\Gamma,\sigma)}
\end{equation}
We also remark that if $\|f\|:=\| f \|_{\ell^1(\Gamma, \sigma)}<1/2$, then, for $P\in \N:$
\begin{equation}\label{eq:series}
\sum_{n= P}^\infty \|f\|^n < \|f\|^P \frac{1}{1-1/2}=2\|f\|^P <\frac{2\|f\|}{2^{P-1}}= \frac{4\|f\|}{2^{P}}.
\end{equation}
Therefore, Equations \eqref{eq:joli1}, \eqref{eq:joli-ineq}  and  \eqref{eq:series} imply, if we denote by  ``$[\sqrt{N}]$" the ``greatest integer less than or equal to $\sqrt{N}$:"

\begin{equation}
\label{eq:joli2}
\begin{split}
\Big|\Big| (1 -  \chi_{N}) \sum_{n=1}^{\infty} f^n \Big|\Big|_{\ell^1(\Gamma, \sigma)} &\leq   \sum_{n=1}^{[\sqrt{N}]} \|(1 -  \chi_{N})f^n \|_{\ell^1(\Gamma, \sigma)} +   \sum_{n=[\sqrt{N}]+1}^{\infty} \| (1 -  \chi_{N}) f^n \|_{\ell^1(\Gamma, \sigma)} \qquad \\
&\leq   \sum_{n=1}^{[\sqrt{N}]} n\|((1 - \chi_{\sqrt{N}})f)\|_{\ell^1(\Gamma,\sigma)}+   \sum_{n=[\sqrt{N}]+1}^{\infty} \| f \|^n \qquad \\
&\leq   [\sqrt{N}]^2\|(1 - \chi_{\sqrt{N}})f \|_{\ell^1(\Gamma, \sigma)} +  \frac{4 \| f \|}{2^{{[\sqrt{N}]+1}}} \qquad \\
&\leq   [\sqrt{N}]^2\|(1 - \chi_{\sqrt{N}})f \|_{\ell^1(\Gamma, \sigma)} +  \frac{4 \| f \|}{2^{{\sqrt{N}}}}. \qquad \\
\end{split}
\end{equation}
So  Equation \eqref{eq:joli2} becomes:

\[
\begin{split}
\Big\| (1 -  \chi_{N}) \sum_{n=1}^{\infty} f^n \Big\|_{\ell^1(\Gamma, \sigma)} 
&\leq   [\sqrt{N}]^2\|(1 - \chi_{\sqrt{N}})f \|_{\ell^1(\Gamma, \sigma)} +  \frac{4 \| f \|_{\ell^1(\Gamma, \sigma)}}{2^{\sqrt{N}}}.\qquad \\
\end{split}
\]
Then, multiplying both sides of the above inequality by $[\sqrt{N}]^{2q}$, for $q\in \N$, we obtain: 
\[
\begin{split}
[\sqrt{N}]^{2q} \Big|\Big| (1 -  \chi_{N}) \sum_{n=1}^{\infty} f^n \Big|\Big|_{\ell^1(\Gamma, \sigma)} 
&\leq   [\sqrt{N}]^{2q+2} \|(1 - \chi_{\sqrt{N}}) f \|+[\sqrt{N}]^{2q} \frac{4 \| f \|_{\ell^1(\Gamma, \sigma)}}{2^{\sqrt{N}}} \qquad \\
&\leq   [\sqrt{N}]^{2q+2} \|(1 - \chi_{[\sqrt{N}]}) f \|_{\ell^1(\Gamma, \sigma)}+N^q \frac{4 \| f \|}{2^{\sqrt{N}}}. \qquad \\
\end{split}
\]
Now recall that $f \in H^{1,\infty}_{\mathbb L}(\Gamma,\sigma)$ implies that $\mu_{2q + 2, \sigma}(f)$ is finite, and that, $\forall N\in\mathbb N$:
$$ [\sqrt{N}]^{2q+2} \|(1 - \chi_{[\sqrt{N}]}) f \|_{\ell^1(\Gamma, \sigma)}\;\leq\;\mu_{2q + 2, \sigma}(f)<\infty.$$
Also $\displaystyle\lim_{N\to\infty}\frac{N^q}{2^{\sqrt{N}}}=0,$ so that the set of nonnegative numbers $\Big\{N^q \frac{4 \| f \|}{2^{\sqrt{N}}}\Big\}_{N=1}^{\infty}$ is bounded above.
Thus
$$\mu_q(\sum_{n=1}^{\infty}f^n)\;<\;\infty,\;\forall q\in\mathbb N,$$ so that $\sum_{n=1}^{\infty}f^n$ is an element of $H^{1,\infty}_{\mathbb L}(\Gamma,\sigma).$   Therefore $\sum_{n=0}^{\infty}f^n=1+\sum_{n=1}^{\infty}f^n$ is an element of $H^{1,\infty}_{\mathbb L}(\Gamma,\sigma)$ as well.   Hence $h=(1-f)$ is invertible in $H^{1,\infty}_{\mathbb L}(\Gamma,\sigma)$  when $\|1-h\|_{\ell^1(\Gamma, \sigma)}<\frac{1}{2}.$

We now consider the case of arbitrary $h\in H^{1,\infty}_{\mathbb L}(\Gamma,\sigma)$ where $h$ is invertible in the larger Banach $^*$-algebra $\ell^1(\Gamma,\sigma),$  but {\it a priori} it is unknown whether or not $h^{-1}\in H^{1,\infty}_{\mathbb L}(\Gamma,\sigma).$   Recall that $1_{\ell^1(\Gamma,\sigma)}=1_{H^{1,\infty}_{\mathbb L}(\Gamma,\sigma)}=1\cdot \delta_{0_\Gamma},$ and $1_{H^{1,\infty}_{\mathbb L}(\Gamma,\sigma)}\;=\;h\ast_{\sigma}h^{-1}.$ Since $H^{1,\infty}_{\mathbb L}(\Gamma,\sigma)$ is dense in $\ell^1(\Gamma,\sigma)$, an adaptation of arguments used by Ji in the proof of Theorem 1.2 in \cite{ji} yields that if we take  $h' \in H^{1,\infty}_{\mathbb L}(\Gamma,\sigma)$ such that
$\|h^{-1}-h'\|_{\ell^1(\Gamma, \sigma)}<\frac{1}{2\cdot \|h\|_{
\ell^1(\Gamma, \sigma)}},$ then:
\begin{equation*}
\| 1- (h *_{\sigma} h')\|_{\ell^1(\Gamma, \sigma)}\;=\;\|(h *_{\sigma}h^{-1})- (h *_{\sigma} h')\|_{\ell^1(\Gamma, \sigma)}\leq\;\|h\|_{\ell^1(h *_{\sigma})}\|\cdot\|h^{-1}-h'\|_{\ell^1(\Gamma, \sigma)}< \frac{1}{2}.
\end{equation*}
We now note that $(h *_{\sigma} h')$ and $1- (h *_{\sigma} h')$ are both elements of $H^{1,\infty}_{\mathbb L}(\Gamma,\sigma).$  By our proof given above, $(h *_{\sigma} h')$ is invertible in $H^{1,\infty}_{\mathbb L}(\Gamma,\sigma),$ with  $(h *_{\sigma} h')^{-1}$  an element of  $H^{1,\infty}_{\mathbb L}(\Gamma,\sigma)$. 
Because $H^{1,\infty}_{\mathbb L}(\Gamma,\sigma)$ is an associative algebra, the element 
$h^{-1} = h' *_{\sigma} (h *_{\sigma} h')^{-1}$  is an element of $H^{1,\infty}_{\mathbb L}(\Gamma,\sigma),$ as well as being an element of $\ell^1(\Gamma,\sigma).$

It follows that if $h\in H^{1,\infty}_{\mathbb L}(\Gamma,\sigma)$ is invertible in $\ell^1(\Gamma,\sigma),$ it is also invertible in $H^{1,\infty}_{\mathbb L}(\Gamma,\sigma).$  Therefore, 
$H^{1,\infty}_{\mathbb L}(\Gamma,\sigma)$ is a subalgebra that is SI in  $\ell^1(\Gamma,\sigma).$

\end{proof}

\subsection{A Noncommutative Wiener's Lemma for Twisted Discrete Nilpotent  Group $C^*$-Algebras}

Because of the equivalence between spectral invariance and being closed under the holomorphic functional calculus  for unital subalgebras of twisted convolution algebras, a statement about the spectral invariance of $\ell^1(\Gamma,\sigma) \subset C^*(\Gamma,\sigma)$ can be viewed as a noncommutative Wiener's Lemma.  Many such results \cite{Aust, Ludwig, Groch-Lein} rely on showing that the spectrum of positive elements in the unital subalgebra is positive.  We will demonstrate that $\ell^1(\Gamma,\sigma)$ belongs to a class of algebras which exhibit this property.

\begin{defn}
[E.g. Ludwig \cite{Ludwig}]
	\label{def:symmetricalgebra}  
    A Banach $^*$-algebra $A$ is called \textbf{symmetric} if for all $a \in A$, $spec_A(a^*a) \subseteq [0, \infty)$.
\end{defn}
As noted by Gr\"ochenig and Leinert in their work on noncommutative Wiener Lemmas \cite{Groch-Lein}, locally compact nilpotent groups have symmetric $L^1$--algebras.  They also cite and build upon the following result due to Ludwig.  
\begin{thm} (\cite{Ludwig}, \cite{Groch-Lein})
	\label{thm:Ludwig}
	If $G$ is a locally compact nilpotent group, then $L^1(G)$ is symmetric.  Consequently, $spec_{L^1(G)}(f) = spec_{B(L^2(G))}(\lambda(f))$ for $f =f^* \in L^1(G),$ and $\lambda$ the regular representation of $L^1(G)$ as bounded operators on $L^2(G).$
\end{thm}
Our twisted discrete  group $C^*$-algebras are in fact quotients of  nilpotent group $C^*$-algebras, as follows.  Let $\Gamma$ be a countable discrete nilpotent  group and $\sigma$ a multiplier on $\Gamma$.  Set $\mathbb{H}_{\sigma} = \Gamma \times \mathbb{T}$.  Equip $\mathbb{H}_{\sigma}$ with multiplication by defining
\begin{equation*}
    (\gamma, z) \cdot_{\sigma} (\gamma', z') = (\gamma + \gamma', zz' \, \sigma(\gamma, \gamma')).
\end{equation*}
Given the product topology, $\mathbb{H}_{\sigma}$ becomes a locally compact  second countable group.  Moreover, $\mathbb{H}_{\sigma}$ is nilpotent. For non-trivial $\sigma$ and abelian groups $\Gamma,\;\mathbb{H}_{\sigma}$ is also in fact a two-step nilpotent group.  The Haar measure $dm$ on $\mathbb{H}_{\sigma}$ can be described explicitly by
\begin{equation*}
    \int_{\mathbb{H}_{\sigma}} F(\gamma, z) \, dm = \sum_{\gamma \in \Gamma} \int_{\mathbb{T}} F(\gamma, z) \, dz.
\end{equation*}
As in the more specific setting of $\Gamma= \mathbb{Z}^d\times \mathbb{Z}^d $ considered by Gr\"ochenig and Leinert in \cite{Groch-Lein}, we define convolution $*_{\sigma}$ of integrable functions on $\mathbb{H}_{\sigma}$ with respect to this measure.  We also follow their terminology in describing $\mathbb{H}_{\sigma}$ as a {\bf Heisenberg-type group}. These groups are also called Mackey groups in \cite{Aust}, where they use the notation $\Gamma_\sigma$ for them. By constructing a unitary group representation of  $\mathbb{H}_{\sigma}$ on $\ell^2(\Gamma),$ it is not hard to see that the twisted group $C^*$-algebra $C^*(\Gamma,\sigma)$ is a quotient of the group $C^*$-algebra $C^*(\mathbb{H}_{\sigma}).$  Note that nilpotent groups are amenable, so that the reduced and full group $C^*$-algebras of $\mathbb{H}_{\sigma}$ are the same.  
In \cite{Aust}, Austad extended Gr\"ochenig and Leinert's arguments to identify conditions for spectral invariance in the more general context of twisted locally compact group $C^*$-algebras.  In both \cite{Aust} and \cite{Groch-Lein}, a condition on the spectrum of self-adjoint elements in $L^1(G)$ for $G$ locally compact with polynomial growth, first obtained by Hulanicki in \cite{Hul} and generalized by B. Barnes in \cite{Bar} (see the Theorem on page 137 of \cite{Hul} and \cite[Theorem 4]{Bar}), played an important role in obtaining these results.  We will now apply Austad's result, together with Ludwig's Theorem, to deduce that there is a noncommutative Wiener's Lemma for twisted discrete nilpotent group $C^*$-algebras.
\begin{lemma}\label{lemma:ncWiener} (c.f. \cite{Aust}, Theorem 3.1)
  Let $\Gamma$ be a  countable discrete nilpotent  group and $\sigma$ a multiplier on $\Gamma$.  Restrict the left-$\sigma$ regular representation $\lambda_{\sigma}: C^*(\Gamma, \sigma) \rightarrow {\it B}(\ell^2(\Gamma))$ to $\ell^1(\Gamma, \sigma)$ and let $\lambda_{\sigma}$ also denote this restriction.  If $f \in \ell^1(\Gamma, \sigma)$, then
  \begin{equation}\label{eq:spec}
      spec_{\ell^1(\Gamma, \sigma)}(f) = spec_{{\it B}(\ell^2(\Gamma))}(\lambda_{\sigma}(f)).
  \end{equation}
In particular, $f$ is invertible in the Banach $^*$-algebra $\ell^1(\Gamma, \sigma)$ if and only if $\lambda_{\sigma}(f)$ is invertible in ${\it  B}(\ell^2(\Gamma))$.   
\end{lemma}
\begin{proof}
Second countable locally compact nilpotent groups are also amenable locally compact groups.  For a twisted group algebra arising as a quotient of a nilpotent (thus amenable) locally compact group and admitting a faithful $^*$-representation, Austad's proof of \cite[Theorem 3.1.(i)]{Aust} gives that the maximal $C^*$-norm is the unique $C^*$-norm.  In fact, this unique $C^*$-norm in the case of $\ell^1(\Gamma, \sigma)$ is completely determined by using $\| \cdot \|_{{\it B}(L^2(\mathbb{H}_{\sigma}, m))}$ and decomposing the left regular representation of $C^*(\mathbb{H}_{\sigma})$
(see Lemma 3.4 and Lemma 3.5 of \cite{Aust}).   Since $\mathbb{H}_{\sigma}$ is nilpotent, Ludwig's Theorem  \ref{thm:Ludwig} implies that $L^1(\mathbb{H}_{\sigma}, m)$ is also symmetric. This shows that the hypotheses of \cite[Theorem 3.1]{Aust} are satisfied, and \cite[Theorem 3.1.(ii)]{Aust} gives Equation \eqref{eq:spec}. 
\end{proof}
$\newline$
The noncommutative Wiener's Lemma above allows us to equip a  twisted discrete nilpotent group $C^*$-algebra with a smooth $\ast$-subalgebra that is spectral invariant, so has  the same $K$-theory as the original $C^*$-algebra.
We have thus achieved the promised generalization of Jolissaint's Proposition (\cite{jol}) to twisted discrete nilpotent group $C^*$-algebras:
\begin{thm} (c.f. Proposition 2.3 of \cite{jol}) \label{prop:disc group multipl} 
	Let $\Gamma$ be a countable discrete nilpotent group and $\sigma$ a multiplier on $\Gamma$.  Suppose that $\mathbb{L}$ is a length function on $\Gamma$.  Then the twisted Fr\'echet $\ast$-subalgebra $H^{1,\infty}_{\mathbb L}(\Gamma,\sigma)$ is dense and has the property of spectral invariance in $C^*(\Gamma, \sigma)$.  Therefore, $H^{1,\infty}_{\mathbb L}(\Gamma,\sigma)$ is stable in $C^*(\Gamma, \sigma)$ under the holomorphic functional calculus. 
\end{thm}
\begin{proof}
Proposition \ref{prop:joliThm} and Lemma \ref{lemma:ncWiener} together give that $f$ is invertible in $H^{1, \infty}_{\mathbb{L}}(\Gamma, \sigma)$ if and only if $\lambda_{\sigma}(f)$ is invertible in $\mathcal{B}(\ell^2(\Gamma))$.  Hence  $H^{1, \infty}_{\mathbb{L}}(\Gamma, \sigma)$ has the property of spectral invariance in $C^*(\Gamma, \sigma),$ i.e., $f$ is invertible in $H^{1, \infty}_{\mathbb{L}}(\Gamma, \sigma)$ if and only if $f$ is invertible in  $C^*(\Gamma, \sigma).$   Since $H^{1, \infty}_{\mathbb{L}}(\Gamma, \sigma)$ contains  $ C_C(\Gamma, \sigma),$ it follows  that $H^{1, \infty}_{\mathbb{L}}(\Gamma, \sigma)$ is dense in $C^*(\Gamma,\sigma).$  By Proposition \ref{prop:frechetalgebra}, $H^{1, \infty}_{\mathbb{L}}(\Gamma, \sigma)$ is a twisted Fr\'echet $\ast$-subalgebra of $C^*(\Gamma,\sigma)$ (see Remark \ref{rmk:twisted Frechet algebra}).  Consequently, by Remark \ref{rmk:Schweizer},  $H^{1,\infty}_{\mathbb L}(\Gamma,\sigma)$ is stable in $C^*(\Gamma, \sigma)$ under the holomorphic functional calculus.  
\end{proof}	
\begin{rmk}
	In contrast to our spectral triple constructions for noncommutative solenoids, where we needed additional properties on  the length functions to ensure that we had Leibniz quantum compact metric spaces, our dense smooth subalgebras do not require the length function to have bounded doubling to obtain spectral invariance for $H^{1, \infty}_{\mathbb{L}}(\Gamma, \sigma)$ in $C^*(\Gamma, \sigma).$ However, our results do rely on the requirement that the group $\Gamma$ be countable, discrete, and nilpotent.
\end{rmk}

We now use the above results to modify our spectral triples for noncommutative solenoids by choosing the dense smooth subalgebras we just constructed.  Fix a prime $p$ and choice of $\theta \in \Omega_p$.  Let $\Gamma = \mathbb{Z}\Big[\frac{1}{p}\Big]\times \mathbb{Z}\Big[\frac{1}{p}\Big]$.  In Section 3, we showed that $(\mathcal{A}_{\theta}^{\mathcal S}, \ell^2(\Gamma), \mathcal{D}_{p})$ with representation $\lambda_{\sigma_{\theta}}$ is a spectral triple for the noncommutative solenoid $\mathcal{A}_{\theta}^{\mathcal S}$.  Recall that for every $\gamma \in \Gamma$,
\begin{equation*}
    L_{\mathcal{D}_p}(\delta_{\gamma}) = \| \, [\mathcal{D}_p, \lambda_{\sigma_{\theta}}(\delta_{\gamma})] \, \|_{\mathcal{B}(\ell^2(\Gamma))} = \mathbb{L}_p^{\Sigma}(\gamma).
\end{equation*}
Thus for every $f \in H_{\mathbb{L}_p^{\Sigma}}^{1, \infty}(\Gamma, \sigma)$, 
\begin{equation*}
     L_{\mathcal{D}_p}(f)= \| \, [\mathcal{D}_p, \lambda_{\sigma_{\theta}}(f)] \, \|_{\mathcal{B}(\ell^2(\Gamma))} = \Big\| \, \Big[\mathcal{D}_p, \sum_{\gamma \in\Gamma} f(\gamma) \, \lambda_{\sigma_{\theta}}(\delta_{\gamma})\Big] \, \Big\|_{\mathcal{B}(\ell^2(\Gamma))} = \Big\| \,  \sum_{\gamma \in \Gamma} f(\gamma) [\mathcal{D}_p, \lambda_{\sigma_{\theta}}(\delta_{\gamma}) ] \, \Big\|_{\mathcal{B}(\ell^2(\Gamma))} 
\end{equation*}
\begin{equation*}
    \leq \sum_{\gamma \in \Gamma} \, |f(\gamma)| \, \| [\mathcal{D}_p, \lambda_{\sigma_{\theta}}(\delta_{\gamma}) ] \, \|_{\mathcal{B}(\ell^2(\Gamma))} = \sum_{\gamma \in \Gamma} \, |f(\gamma)|  \mathbb{L}_p^{\Sigma}(\gamma) \leq \sum_{\gamma \in \Gamma} \, |f(\gamma)| (1 + \mathbb{L}_p^{\Sigma}(\gamma)) = \| f \|_{1, 1, \mathbb{L}_p^{\Sigma}} < \infty.
\end{equation*}
Induction can then be used to verify that for $k \in \mathbb{N}$,
\begin{equation*}
\| \, \underbrace{[\mathcal{D}_p, [\mathcal{D}_p, \cdots [\mathcal{D}_p, \lambda_{\sigma_{\theta}}(f)] \cdots], ] }_{k} \, \|_{\mathcal{B}(\ell^2(\Gamma))} \leq \| f \|_{1, k, \mathbb{L}_p^{\Sigma}} < \infty.
\end{equation*}
Taking the point of view that the iterated commutants above can be seen as a noncommutative form of higher-order differentiation, the above calculations validate our choice of  $H_{\mathbb{L}_p^{\Sigma}}^{1, \infty}(\Gamma, \sigma)$ as a smooth subalgebra in $C^*(\Gamma, \sigma).$  
With the selection of  $H_{\mathbb{L}_p^{\Sigma}}^{1, \infty}(\Gamma, \sigma)$ as  dense smooth subalgebra, our spectral triples for noncommutative solenoids can support further study of these inductive limit $C^*$-algebras as noncommutative manifolds.  We now are able to adjust the results of Theorem \ref{thm:ncsolenoidspectraltriple} using a dense subalgebra that has the property of spectral invariance.

\begin{cor}\label{cor: main thm}
Fix a prime $p$.  Let $\Gamma = \mathbb{Z}\Big[\frac{1}{p}\Big]\times \mathbb{Z}\Big[\frac{1}{p}\Big]$.  For every $\theta \in \Omega_p$, $(\mathcal{A}_{\theta}^{\mathcal S}, \ell^2(\Gamma), \mathcal{D}_{p})$ with representation $\lambda_{\sigma_{\theta}}$ is a spectral triple for the noncommutative solenoid $C^*(\Gamma,\sigma_{\theta})\;=\;\mathcal{A}_{\theta}^{\mathcal S},$ with  associated smooth subalgebra  $H_{\mathbb{L}_p^{\Sigma}}^{1, \infty}(\Gamma, \sigma_{\theta})$.  Furthermore, the twisted Fr\'echet $\ast$-subalgebra $H_{\mathbb{L}_p^{\Sigma}}^{1, \infty}(\Gamma, \sigma_{\theta})$ is a proper dense subalgebra of  $C^*(\Gamma,\sigma_{\theta})\;=\;\mathcal{A}_{\theta}^{\mathcal S}$ that is stable under the holomorphic functional calculus. 
\end{cor}
\begin{proof}
Theorem \ref{prop:disc group multipl} gives that $H_{\mathbb{L}_p^{\Sigma}}^{1, \infty}(\Gamma, \sigma_{\theta})$ is a dense twisted Fr\'echet $\ast$-subalgebra that is stable in $C^*(\Gamma, \sigma_{\theta})$ under the holomorphic functional calculus.  Since  $[\mathcal{D}_p, \lambda_{\sigma_{\theta}}(f)] \in \mathcal{B}(\ell^2(\Gamma))$ for every $f \in H_{\mathbb{L}_p^{\Sigma}}^{1, \infty}(\Gamma, \sigma_{\theta})$, $H_{\mathbb{L}_p^{\Sigma}}^{1, \infty}(\Gamma, \sigma_{\theta})$ is also a proper dense subalgebra of $C^*(\Gamma,\sigma_{\theta})\;=\;\mathcal{A}_{\theta}^{\mathcal S}$. 
\end{proof}

We also revisit and enhance the countable inductive system of finite-dimensional spectral triples considered in Theorem \ref{thm:FGinductivelim}.   For each $j \in \mathbb{N}$, $(C^*(\mathbb{Z}^2, \sigma_{\theta_{2j}}), \ell^2(\Gamma_j), D_{p,j})$ is a spectral triple for the twisted discrete abelian group $C^*$-algebra $C^*(\mathbb{Z}^2, \sigma_{\theta_{2j}})$ coming from the length function $\mathbb{L}_{p,j}^\Sigma.$  Theorem \ref{prop:disc group multipl} gives that $H_{\mathbb{L}_{p,j}^{\Sigma}}^{1, \infty}(\Gamma_j, (\sigma_{\theta})_j)$ is a proper dense twisted Fr\'echet $\ast$-subalgebra that is stable in $C^*(\Gamma_j, (\sigma_{\theta})_j)$ under the holomorphic functional calculus.  Recall that the representation $\pi_{\theta_j}$ is defined using a $^*$-isomorphism $\upsilon_j$ which takes $C^*(\mathbb{Z}^2, \sigma_{\theta_{2j}})$ to $C^*(\Gamma_j, (\sigma_{\theta})_j)$.  Our inductive system of $C^*$-algebras 
\begin{equation*}\label{eq:beta}
 \left\{ (C^*(\Gamma_j, (\sigma_{\theta})_j, \beta_{j,k}) \right\}_{j \in \mathbb{N}}   
\end{equation*}
therefore supports the inductive system of twisted Fr\'echet $\ast$-algebras 
\begin{equation*}
 \left\{ (H_{\mathbb{L}_{p,j}^{\Sigma}}^{1, \infty}(\Gamma_j, (\sigma_{\theta})_j), \beta_{j,k}) \right\}_{j \in \mathbb{N}}.     
\end{equation*}
Furthermore, each $H_{\mathbb{L}_{p,j}^{\Sigma}}^{1, \infty}(\Gamma_j, (\sigma_{\theta})_j)$ is $^*$-isomorphic to a proper dense subalgebra of the corresponding $C^*$-algebra in our inductive system of rotation algebras
\begin{equation*}
\left\{ (C^*(\mathbb{Z}^2, \sigma_{\theta_{2j}}), \phi_{j,k}) \right\}_{j \in \mathbb{N}}.   
\end{equation*}
In particular, each $\upsilon_{j}^{-1}\Big(H_{\mathbb{L}_{p,j}^{\Sigma}}^{1, \infty}(\Gamma_j, (\sigma_{\theta})_j\Big)$ can be explicitly identified with the twisted Fr\'echet $\ast$-algebra $H_{\mathbb{L}_{p,j}^{\Sigma}\circ\upsilon_{j} }^{1, \infty}(\mathbb{Z}^2, \sigma_{\theta_{2j}}),$ and thus is stable under the holomorphic functional calculus in $C^*(\mathbb{Z}^2, \sigma_{\theta_{2j}})$.  Hence our inductive system of twisted Fr\'echet $\ast$-algebras encodes the respective $K$-theories of our inductive system of rotation algebras.  Thus  $H_{\mathbb{L}_{p}^{\Sigma}}^{1, \infty}(\Gamma, \sigma_{\theta})$ can be expressed as the inductive limit of our inductive system of twisted Fr\'echet $\ast$-algebras  $\{H_{\mathbb{L}_{p,j}^{\Sigma}\circ\upsilon_{j} }^{1, \infty}(\mathbb{Z}^2, \sigma_{\theta_{2j}})\}_{j\in \mathbb N}.$  Since noncommutative solenoids can be identified with inductive limits of rotation algebras, we hope that equipping our inductive system of spectral triples for rotation algebras with the corresponding inductive system of twisted Fr\'echet $\ast$-algebras as the underlying smooth subalgebras in the limiting case will enrich the development, study, and understanding of the noncommutative geometry of noncommutative solenoids.

\end{document}